\documentclass[11pt]{article}
\usepackage{amsmath,amsthm}
\usepackage{amssymb,mathrsfs}
\usepackage{authblk}
\usepackage{color}
\usepackage{geometry}
\usepackage{graphicx}
\usepackage{verbatim}
\usepackage[format=hang,font=small,textfont=it]{caption}
\usepackage[nottoc]{tocbibind}
\usepackage{labelschanged}

\newcommand\nnfootnote[1]{   
\begin{NoHyper}
\renewcommand\thefootnote{}\footnote{#1}
\addtocounter{footnote}{-1}
\end{NoHyper}}

\allowdisplaybreaks[4]

\usepackage{hyperref}
\hypersetup{colorlinks, linkcolor=blue}
\usepackage{soul}
\soulregister\cite7
\soulregister\eqref7
\soulregister\eqref7
\hypersetup{colorlinks, linkcolor=blue}

\newtheorem{thm}{Theorem}[section]
\newtheorem{lemma}[thm]{Lemma}
\newtheorem{prop}[thm]{Propsition}

\newtheorem{thmx}{Theorem}
\theoremstyle{definition}
\newtheorem{definition}[thm]{Definition}

\newtheorem{remark}[thm]{Remark}

\newcommand\ddd{\mathrm{d}}

\newcommand\supp{\mathrm{supp}}
\newcommand\holder{Hölder's~}
\newcommand\schrodinger{Schrödinger~}

\newcommand\bR{\mathbb{R}}

\newcommand\bN{\mathbb{N}}
\newcommand\bZ{\mathbb{Z}}

\newcommand\bS{\mathbb{S}}

\def \l {\left}
\def \r {\right}

\begin{document}
\title{Extremals for $\alpha$-Strichartz inequalities}

\setcounter{footnote}{0}
\author{Boning Di, Dunyan Yan}
\date{}
\maketitle
\nnfootnote{{2020 \emph{Mathematics Subject Classification}: Primary 42A38; Secondary 35B38, 35Q41.}}
\nnfootnote{\emph{Key words and phases}: Sharp Fourier restriction theory, extremals, $\alpha$-Strichartz inequalities, profile decomposition.}
\begin{abstract}
A necessary and sufficient condition on the precompactness of extremal sequences for one dimensional $\alpha$-Strichartz inequalities, equivalently $\alpha$-Fourier extension estimates, is established based on the profile decomposition arguments. One of our main tools is an operator-convergence \textit{dislocation property} consequence which comes from the van der Corput Lemma. Our result is valid in asymmetric cases as well. In addition, we obtain the existence of extremals for non-endpoint $\alpha$-Strichartz inequalities.
\end{abstract}

\section{Introduction}
For $\alpha>1$, we investigate the following symmetric $\alpha$-Strichartz inequality
\begin{equation}\label{equation_symmetric alpha-Strichartz}
\l\|[D^{\frac{\alpha-2}{6}}][e^{it|\nabla|^{\alpha}}]u\r\|_{L_{t,x}^6(\bR^{2})}\leq \mathbf{M}_{\alpha} \|u\|_{L_x^2(\bR)},
\end{equation}
where
\[\mathbf{M}_{\alpha}:=\sup\l\{\l\|[D^{\frac{\alpha-2}{6}}][e^{it|\nabla|^{\alpha}}]u\r\|_{L_{t,x}^{6}(\bR^{2})}: \|u\|_{L_x^2(\bR)}=1\r\}\]
is the sharp constant and
\[[e^{it|\nabla|^{\alpha}}]u(x):=\mathscr{F}^{-1}e^{-it|\xi|^{\alpha}}\mathscr{F}[u](x),\quad [D^s]u(x):=\mathscr{F}^{-1}|\xi|^s\mathscr{F}[u](x), \quad \mathscr{F}[u](\xi):=\int_{\bR} e^{-ix\xi}u(x)\ddd x,\]
with $\mathscr{F}$ denoting the spatial Fourier transform. This estimate \eqref{equation_symmetric alpha-Strichartz} comes from Kenig et al. \cite[Theorem 2.3]{KPV_1991} which is also named Fourier extension estimate. Moreover it says that, for every $\alpha>1$, there holds the (mixed norm) asymmetric $\alpha$-Strichartz inequality
\begin{equation}\label{equation_asymmetric alpha-Strichartz}
\l\|[D^{\frac{\alpha-2}{q}}][e^{it|\nabla|^{\alpha}}]u\r\|_{L_t^q L_x^r(\bR^2)} \leq \tilde{\mathbf{M}}_{\alpha, q,r} \|u\|_{L_x^2(\bR)},
\end{equation}
where $2/q+1/r=1/2$ with the sharp constant $\tilde{\mathbf{M}}_{\alpha, q,r}$ defined by
\[\tilde{\mathbf{M}}_{\alpha,q,r}:=\sup\l\{\l\|[D^{\frac{\alpha-2}{q}}][e^{it|\nabla|^{\alpha}}]u\r\|_{L_t^q L_x^r(\bR^2)}: \|u\|_{L_x^2(\bR)}=1\r\}.\]
We call the pairs $(q,r)=(\infty,2)$ and $(q,r)=(4,\infty)$ \textit{endpoint pairs}. Otherwise the pairs $(q,r)$ are called \textit{non-endpoint pairs}.

The \textit{symmetries} for these $\alpha$-Strichartz inequalities, on the $L_x^2$ side, are time-space translations and scaling as follows
\[[g_n^{\mathrm{sym}}]u:=[e^{it_n|\nabla|^{\alpha}}]\l[(h_n)^{-1/2}u\l(\frac{\cdot-x_n}{h_n}\r)\r], \quad (h_n,x_n,t_n)\in \bR_{+}\times \bR\times\bR;\]
and the \textit{associated group} $G^{\mathrm{sym}}$ is defined by
\[G^{\mathrm{sym}}:=\Big\{[g_n^{\mathrm{sym}}]: (h_n,x_n,t_n)\in \bR_{+}\times \bR\times\bR\Big\}.\]
To state the results more precisely, we say a sequence of functions $(f_n)$ in $L^2(\bR)$ is \textit{precompact up to symmetries} if there exists a sequence of symmetries $([g_n^{\mathrm{sym}}])$ in $G^{\mathrm{sym}}$ such that $\l([g_n^{\mathrm{sym}}]f_n\r)$ is precompact in $L^2(\bR)$. On the other hand, a sequence of functions $(f_n)$ in $L^2(\bR)$ \textit{concentrates at a point} $x_0\in\bR$ if for arbitrary $\varepsilon,\rho>0$, there exists $N\in\bN_{+}$ such that for every $n>N$, there holds
\[\int_{|x-x_0|\geq\rho}|f_n(x)|^2\ddd x\leq \varepsilon \|f_n\|_{L^2(\bR)}^2.\]
Meanwhile a sequence of functions $(f_n)$ in $L^2(\bR)$ is an \textit{extremal sequence} for $\tilde{\mathbf{M}}_{\alpha, q,r}$ if it satisfies
\[\|f_n\|_{L^2(\bR)}=1,\quad \lim_{n\to\infty} \l\|[D^{\frac{\alpha-2}{6}}] [e^{it|\nabla|^{\alpha}}] f_n\r\|_{L_{t,x}^{6}(\bR^{2})}=\tilde{\mathbf{M}}_{\alpha, q,r}.\]
The sharp Fourier restriction theory, more generally the sharp constant theory, has been an important part in harmonic analysis. Readers are referred to the survey \cite{FO_2017} and the references therein for some recent progress on sharp Fourier restriction theory. One of the more recent results is
\begin{thmx}[\cite{BOQ_2020}]\label{Thm_extremals-symmetric}
All the extremal sequences for $\mathbf{M}_{\alpha}$ are precompact up to symmetries if and only if
\begin{equation}\label{Thm_extremals-symmetric_1}
\mathbf{M}_{\alpha}>\l[\sqrt{3}\alpha(\alpha-1)\r]^{-\frac{1}{6}}.
\end{equation}
In particular, if the strict inequality \eqref{Thm_extremals-symmetric_1} holds, then there exists an extremal for $\mathbf{M}_{\alpha}$. If on the contrary the equality holds in \eqref{Thm_extremals-symmetric_1}, then given any $x_0\in \bR$, there exists an extremal sequence for $\mathbf{M}_{\alpha}$ which concentrates at $x_0$.
\end{thmx}
This result is previously obtained by Brocchi et al. \cite[Theorem 1.3]{BOQ_2020}. The proof there uses a variant of Lions' \textit{concentration-compactness} lemma from \cite{Lions_1984a, Lions_1984b} together with a variant of Brézis-Lieb lemma from \cite{BL_1983, Lieb_1983}. As pointed out in \cite{BOQ_2020}, various of results with a similar condition to \eqref{Thm_extremals-symmetric_1} have been studied in recent literature. In our paper, this condition comes from the asymptotic \schrodinger behavior Lemma \ref{Lemma_asymptotic schrodinger}\footnote{In view of Frank-Sabin \cite[Remark 2.6]{FS_2018}, this behavior may also be called \textit{approximate symmetries}.}, see also Remark \ref{Remark_asymptotic schrodinger} and Remark \ref{Remark_asymptotic schrodinger_asymmetric}. Roughly speaking, to get the existence of extremals, there may be some strict inequality conditions like \eqref{Thm_extremals-symmetric_1} to rule out some concentrate-type situations which deduce the loss of compactness. We refer to \cite{CS_2012_A&P, FLS_2016, FS_2018} for more discussions on these type of conditions in the low dimensional sphere and cubic curve cases.

The main purpose of this article is investigating the extremal problems for $\alpha$-Strichartz inequalities by means of \textit{profile decomposition arguments}. One of our results, Theorem \ref{Thm_extremals asymmetric} below, generalizes the aforementioned Theorem \ref{Thm_extremals-symmetric} to asymmetric cases. Furthermore as an application of our profile decomposition consequences, for $\alpha\geq2$, we also give the existence of extremals for \textit{non-endpoint $\alpha$-Strichartz inequalities} \eqref{equ_nonendpoint alpha-Strichartz} which will be presented later\footnote{We follow this \textit{non-endpoint} terminology from Hundertmark-Shao \cite{HS_2012}.}. The key ingredient here to establish this generalized profile decomposition Proposition \ref{Prop_linear profile decomposition} is a \textit{conditional dislocation property} consequence Proposition \ref{Prop_dislocation property} on the weak operator topology convergence for some $L^2$-unitary operators. Now we state our first main result as follows.
\begin{thm}\label{Thm_extremals asymmetric}
For the non-endpoint pairs $(q,r)$, all the extremal sequences for $\tilde{\mathbf{M}}_{\alpha, q,r}$ are precompact up to symmetries if and only if
\begin{equation}\label{Thm_extremals asymmetric_1}
\tilde{\mathbf{M}}_{\alpha, q,r}>\l(\frac{\alpha^2-\alpha}{2}\r)^{-\frac{1}{q}} \tilde{\mathbf{M}}_{2,q,r}.
\end{equation}
In particular, if the strict inequality \eqref{Thm_extremals asymmetric_1} holds, then there exists an extremal for $\tilde{\mathbf{M}}_{\alpha, q,r}$. If on the contrary the equality holds in \eqref{Thm_extremals asymmetric_1}, then given any $x_0\in \bR$, there exists an extremal sequence for $\tilde{\mathbf{M}}_{\alpha, q,r}$ which concentrates at $x_0$.
\end{thm}

As we have mentioned above, Theorem \ref{Thm_extremals asymmetric} extends the previous result \cite[Theorem 1.3]{BOQ_2020}. Meanwhile by taking some symmetries\footnote{Notice the construction of $\tilde{u}_n$ in the proof of Theorem \ref{Thm_extremals-symmetric} in Section \ref{Sec_Symmetric alpha-Strichartz extremals}.}, on the Fourier side, Theorem \ref{Thm_extremals asymmetric} claims that if the equality holds in \eqref{Thm_extremals asymmetric_1} then there exists an extremal sequence which concentrates at one fixed frequency. Thereby this Theorem \ref{Thm_extremals asymmetric}, in some sense, also coincides with the result in \cite{FS_2018} where the extremal sequence concentrates at two opposite frequencies due to some symmetries of the odd curves.

Here we make some historical remarks first. For the case $\alpha=2$ in \eqref{equation_symmetric alpha-Strichartz}, the classical Strichartz inequality (Stein-Tomas inequality for the paraboloid), abundant conclusions have been made: the existence of extremals is proved by Kunze \cite{Kunze_2003} for one dimensional case and by Shao \cite{Shao_2009_EJDE} for general dimensions; in low dimensions, up to symmetries, the only extremals are shown to be Gaussians by Foschi \cite{Foschi_2007} and Hundertmark-Zharnitsky \cite{HZ_2006} independently. Extremals are conjectured to be Gaussians in all dimensions \cite{HZ_2006}. Meanwhile, on the Stein-Tomas inequality for the sphere, we briefly mention that Christ-Shao \cite{CS_2012_A&P, Shao_2016} give the existence of extremals in low dimensions and Foschi \cite{Foschi_2015} shows that the extremals are constants for two-dimension sphere $\bS^2$. We refer to \cite{FO_2017, OQ_2021} and the references therein for more recent results on the sharp Fourier restriction theory in the sphere situation.

As for the case $\alpha=4$ in \eqref{equation_symmetric alpha-Strichartz}, Jiang et al. \cite{JPS_2010, JSS_2017} give some dichotomy results on the existence of extremals by using the profile decomposition from \cite{BG_1999, BV_2007, CK_2007, Keraani_2001, MV_1998}. For more general case $\alpha>1$ in \eqref{equation_symmetric alpha-Strichartz}, Brocchi et al. \cite{BOQ_2020} resolves the dichotomy in \cite{JPS_2010} by using a \textit{geometric comparison principle} developed in \cite{OQ_2018} which resolves the dichotomy in \cite{JSS_2017}. As far as we know, there is no extremal result on the asymmetric $\alpha$-Strichartz inequality \eqref{equation_asymmetric alpha-Strichartz} with general $\alpha>1$, except for the classical $\alpha=2$ case in \eqref{equation_asymmetric alpha-Strichartz} which has been studied in some papers such as \cite{BBCH_2009, Carneiro_2009, Goncalves_2019, Shao_2009_EJDE}. Meanwhile, it should be mentioned that Frank-Sabin \cite{FS_2018} has studied the existence of extremals for Airy-Strichartz inequality (odd cubic curve), whose result is also valid for non-endpoint asymmetric cases, by using the \textit{missing mass method}.

Note that $\alpha>1$ may not be a natural number in our setting and this fact leads to some barriers. In order to establish the desired linear profile decomposition, one of the main results we should establish is the conditional dislocation property Proposition \ref{Prop_dislocation property} for some unitary operators on $L^2(\bR)$. We begin with the definitions for the \textit{dislocation group} and the $L^2$-unitary operators that, maybe non-compact, we are concerned about. For parameters $(h_0, x_0, \xi_0, \theta_0)\in  \bR_{+} \times \bR^d \times \bR^d\times \bR$, the unitary operators $g_0$ on $L_x^2(\bR^d)$ is defined by
\[[g_0]\phi(x):=g_{\theta_0, \xi_0, x_0, h_0}[\phi](x):=e^{i\theta_0}h_0^{-\frac{d}{2}}e^{i x\cdot \xi_0}\phi\l(\frac{x-x_0}{h_0}\r).\]
We should point out that the parameter $\theta_0$ is inessential and we use it just because it may be deduced from other parameters on the Strichartz space.
\begin{definition}[Dislocation group \cite{ST_2002}]
Let $H$ be a separable Hilbert space and let $G$ be a group of unitary operators on $H$. We said $G$ is a \textit{group of dislocations} if it satisfies the following condition: for every sequence $([g_n])\subset G$ does not converge weakly (in weak operator topology) to zero, there exists a renamed strongly convergent subsequence of $([g_n])$ such that the strong limit (in strong operator topology) is not zero.
\end{definition}

In the classical case $\alpha=2$, due to the Galilean invariance of classical \schrodinger equations, the dislocation property for the group generated by non-compact $L^2$-unitary operators is obvious. This potentially crucial fact, when establishing the classical profile decomposition, deduces the orthogonality of these decomposed profiles in Strichartz spaces. Hence it is a natural idea to generalize this dislocation property to the $\alpha$-Strichartz setting. On the other hand, we may do some adaption along the way we generalize it. The following conditional dislocation property proposition comes from the method of stationary phase, which is contained in \cite[Chapter 8]{Stein_1993} and \cite[Chapter 6]{Wolff_2003}, or more precisely the classical van der Corput Lemma \cite[p. 332, Proposition 2]{Stein_1993}.
\begin{prop}[Conditional dislocation property]\label{Prop_dislocation property}
When $d=1$, if we assume that for fixed $j\neq k$ either
\[\lim_{n\to\infty}\l(\frac{h_n^j}{h_n^k}+\frac{h_n^k}{h_n^j}+(h_n^j+h_n^k)\l|\xi_n^j-\xi_n^k\r|\r)=\infty\]
or $(h_n^j,\xi_n^j)\equiv (h_n^k,\xi_n^k)$. Then the group $G$, generated by the $L^2$-symmetries
\begin{equation}
[g_n^j]^{-1}[e^{it_n^j|\nabla|^{\alpha}}][e^{-it_n^k|\nabla|^{\alpha}}][g_n^k],
\end{equation}
is a group of dislocations provided that $\alpha\in \bZ_{+}$ and $\alpha\geq2$. Moreover, the group $G$ is a group of dislocations for all real numbers $\alpha>1$.
\end{prop}
\begin{remark}
As we shall see later in Lemma \ref{Lemma_frequency and scaling}, the assumptions in Proposition \ref{Prop_dislocation property} arise naturally during the construction of linear profile decomposition. Analogous assumptions can also be seen in \cite[Theorem 1.3]{JPS_2010} and \cite[p. 10]{JSS_2017} as well as some earlier papers such as \cite{BG_1999, CK_2007, Keraani_2001}. We will use this conditional dislocation property to describe the orthogonality for profiles instead of using the parameters therein, since we need to deal with the situation that $\alpha$ is not a natural number.
\end{remark}

The profile decomposition results are intensively studied and widely used in many topics. Besides some of the aforementioned references such as \cite{BG_1999, BV_2007, CK_2007, Keraani_2001, MV_1998} which establish these profile decompositions in different analysis situations, the profile decomposition may also be called \textit{bubble decomposition} in the literature due to some geometric background. We refer to \cite[p. 359]{KV_2013} for a historical discussion, see also \cite[p. 373]{KV_2013}. Here with the conditional dislocation property Proposition \ref{Prop_dislocation property} in place, we are able to show the following $\alpha$-Strichartz version linear profile decomposition.
\begin{prop}[Linear profile decomposition for the $\alpha$-Strichartz version]\label{Prop_linear profile decomposition}
Let $(u_n)$ be a bounded sequence in $L^2(\bR)$. Then, up to subsequences, there exist a sequence of operators $([T_n^j])$ defined by
\[[T_n^{j}]\phi(x):=[e^{-it_n^j|\nabla|^{\alpha}}]\l[(h_n^j)^{-\frac{1}{2}}e^{i(x-x_n^j)\xi_n^j}\phi\l(\frac{x-x_n^j}{h_n^j}\r)\r]\]
with $(h_n^j, x_n^j, \xi_n^j, t_n^j) \in \bR_{+}\times\bR\times\bR\times\bR$ and a sequence of functions $(\phi^j)\subset L^2(\bR)$ such that for every $J\geq1$, we have the profile decomposition
\begin{equation}\label{Prop_linear profile decomposition_1}
u_n=\sum_{j=1}^{J} [T_n^j]\phi^j+\omega_n^{J},
\end{equation}
where the decomposition possesses the following properties: firstly the remainder term $\omega_n^{J}$ has vanishing Strichartz norm
\begin{equation}\label{Prop_linear profile decomposition_2}
\lim_{J\to\infty}\limsup_{n\to\infty}\l\|[D^{\frac{\alpha-2}{6}}][e^{it|\nabla|^{\alpha}}]\omega_n^{J}\r\|_{L_{t,x}^6(\bR^{2})}=0;
\end{equation}
secondly the sequence of operators $[T_n^j]$ satisfies that if $j\neq k$, there holds the limit-orthogonality property
\begin{equation}\label{Prop_linear profile decomposition_3}
[T_n^k]^{-1}[T_n^j]\rightharpoonup 0
\end{equation}
as $n$ goes to infinity in the weak operator topology of $\mathcal{B}(L^2)$; moreover for each $J\geq1$, we have
\begin{equation}\label{Prop_linear profile decomposition_4}
\lim_{n\to\infty}\l[\|u_n\|_{L^2(\bR)}^2-\l(\sum_{j=1}^{J}\|\phi^j\|_{L^2(\bR)}^2\r)-\|\omega_n^{J}\|_{L^2(\bR)}^2\r]=0.
\end{equation}
\end{prop}

\begin{remark}\label{Remark_linear profile decomposition}
We should point out that the limit-orthogonality \eqref{Prop_linear profile decomposition_3} of the operators $[T_n^j]$ is crucial and powerful, especially when combined with the conditional dislocation property Proposition \ref{Prop_dislocation property}. By the $L^2$-almost orthogonal identity \eqref{Prop_linear profile decomposition_4}, it can be easily inferred that for every $j\neq k$ in Proposition \ref{Prop_linear profile decomposition}, there holds
\begin{equation}\label{Remark_linear profile decomposition_1}
\lim_{n\to\infty}\l\langle[T_n^j]\phi^j, [T_n^k]\phi^k\r\rangle_{L^2(\bR)}=0;
\end{equation}
and for each $j\leq J$, there holds
\begin{equation}\label{Remark_linear profile decomposition_2}
\lim_{n\to\infty}\l\langle[T_n^j]\phi^j, \omega_n^{J}\r\rangle_{L^2(\bR)}=0.
\end{equation}
\end{remark}

Meanwhile, the $\alpha$-Strichartz version profile decomposition Proposition \ref{Prop_linear profile decomposition} is equipped with the following Strichartz-orthogonality for the decomposed linear profiles.
\begin{prop}[Strichartz-orthogonality of profiles]\label{Prop_Strichartz-orthogonal profiles}
Furthermore, in the linear profile decomposition Proposition \ref{Prop_linear profile decomposition}, for $j\neq k$ there holds
\begin{equation}\label{Prop_Strichartz-orthogonal profiles_1}
\lim_{n\to\infty}\l\|[D^{\frac{\alpha-2}{6}}][e^{it|\nabla|^{\alpha}}][T_n^j]\phi^j \cdot [D^{\frac{\alpha-2}{6}}][e^{it|\nabla|^{\alpha}}][T_n^k]\phi^k\r\|_{L_{t,x}^{3}(\bR^{2})}=0.
\end{equation}
Thus for each $J\geq1$, by \holder inequality, there holds
\begin{equation}\label{Prop_Strichartz-orthogonal profiles_2}
\limsup_{n\to\infty}\l(\l\|\sum_{j=1}^J[D^{\frac{\alpha-2}{6}}][e^{it|\nabla|^{\alpha}}][T_n^j]\phi^j\r\|_{L_{t,x}^6}^{6} -\sum_{j=1}^J\l\|[D^{\frac{\alpha-2}{6}}][e^{it|\nabla|^{\alpha}}][T_n^j]\phi^j\r\|_{L_{t,x}^6}^{6}\r)=0.
\end{equation}
\end{prop}

Finally, as an application of our profile decomposition results, we turn to the following estimates. For $\alpha\geq2$, the result of Kenig et al. \cite[Theorem 2.3]{KPV_1991} and Sobolev inequalities imply the following \textit{non-endpoint $\alpha$-Strichartz estimates}
\begin{equation}\label{equ_nonendpoint alpha-Strichartz}
\l\|[e^{it|\nabla|^{\alpha}}]u\r\|_{L_{t,x}^{2\alpha+2}(\bR^2)}\leq \dot{\mathbf{M}}_{\alpha} \|u\|_{L_x^2(\bR)},
\end{equation}
where $\dot{\mathbf{M}}_{\alpha}$ is the sharp constant
\[\dot{\mathbf{M}}_{\alpha}:=\sup\l\{\l\|[e^{it|\nabla|^{\alpha}}]u\r\|_{L_{t,x}^{2\alpha+2}(\bR^{2})}: \|u\|_{L_x^2(\bR)}=1\r\}.\]
See, for instance, \cite[Theorem 2.4]{KPV_1991} for analogous arguments. In \cite{HS_2012}, Hundertmark and Shao give the existence of extremals for some similar non-endpoint Airy-Strichartz inequalities based on the profile decomposition of Airy-Strichartz version. Moreover they also establish the analyticity of these extremals on the Fourier space by using a bootstrap argument. In the spirit of their work and based on the generalized profile decomposition consequences obtained above, we show the existence of extremals for $\dot{M}_{\alpha}$ as a short incidental result.
\begin{thm}\label{Thm_extremals nonendpoint alpha-Strichartz}
For every $\alpha\geq 2$, there exists an extremal for $\dot{\mathbf{M}}_{\alpha}$.
\end{thm}

The outline of this paper is as follows. In Section \ref{Sec_Dislocation from van der Corput} we begin with proving the conditional dislocation property Proposition \ref{Prop_dislocation property} which is one of the key ingredients in our paper. Then we extract the frequency and scaling parameters for the desired $\alpha$-Strichartz version linear profile decomposition in Section \ref{Sec_First-step}. After that, by using Proposition \ref{Prop_dislocation property}, we are able to obtain the time and space translation parameters in Section \ref{Sec_Second-step} and further present the $\alpha$-Strichartz version linear profile decomposition in Section \ref{Sec_Profile decomposition}. Then Section \ref{Sec_Symmetric alpha-Strichartz extremals} and Section \ref{Sec_Asymmetric alpha-Strichartz} contain the extremal results for symmetric $\alpha$-Strichartz estimates Theorem \ref{Thm_extremals-symmetric} and asymmetric $\alpha$-Strichartz estimates Theorem \ref{Thm_extremals asymmetric} respectively. Finally the proof of Theorem \ref{Thm_extremals nonendpoint alpha-Strichartz} is provided in Section \ref{Sec_Extremals nonendpoint alpha-Strichartz}.

We end this section with some notations. Firstly we use the familiar notation $x\lesssim y$ to denote that there exists a finite constant $C$ such that $|x|\leq C|y|$, similarly for $x\gtrsim y$ and $x\sim y$. Sometimes we may show the dependence such as $x\lesssim_{\alpha} y$ for the constant $C=C(\alpha)$ if necessary. Occasionally we may write $\hat{u}:=\mathscr{F}[u]$ or $u^{\wedge}:=\mathscr{F}[u]$, similarly for the inverse Fourier transform $\check{u}=u^{\vee}:=\mathscr{F}^{-1}[u]$. In addition, since there may be different topologies throughout this paper, we use the notation $\to$ to denote strong convergence and the notation $\rightharpoonup$ to denote weak convergence. More precisely, for a sequence of functions $(f_n)\subset L^p$, we write $f_n\to f_0$ for the fact that $f_n$ converge to $f_0$ as $n$ goes to infinity in the norm (strong) topology of $L^p$, and write $f_n \rightharpoonup f_0$ for the fact that $f_n$ converge to $f_0$ as $n$ goes to infinity in the weak topology of $L^p$. As for a sequence of operators $([T_n])$ on the space $H$ which means $([T_n])\subset \mathcal{B}(H)$, similarly $[T_n]\to [T_0]$ and $[T_n]\rightharpoonup [T_0]$ denote the convergence in the strong operator topology and weak operator topology of $\mathcal{B}(H)$ respectively.

\section{Dislocation property from van der Corput Lemma} \label{Sec_Dislocation from van der Corput}
Before to give the linear profile decomposition, we show the conditional dislocation property Proposition \ref{Prop_dislocation property} first since it will be used in the forthcoming work of extracting time-space translation parameters in Section \ref{Sec_Second-step}. As what we have said before this property is, in some sense but not directly, generalization of the classical \schrodinger dislocation property which comes from the Galilean invariance. Note that the conditional dislocation property Proposition \ref{Prop_dislocation property} has been adapted to the desired profile decomposition Proposition \ref{Prop_linear profile decomposition} when we establish it.

\begin{proof}[\textbf{Proof of Proposition \ref{Prop_dislocation property}}]
We begin with proving the first conclusion. By a standard approximation argument together with the symmetry of $j$ and $k$, it suffices to prove that if
\begin{equation}\label{equation_dislocations_1}
\lim_{n\to\infty}\l\langle [g_n^j]^{-1}[e^{it_n^j|\nabla|^{\alpha}}][e^{-it_n^k|\nabla|^{\alpha}}][g_n^k]\phi, \psi\r\rangle\neq0
\end{equation}
for some Schwartz functions $\phi$ and $\psi$ whose Fourier supports are compact, then there exist a unitary operator $G^{jk}\in \mathcal{B}(L_x^2)$ and a subsequence for $n$ (also denoted by $n$) such that
\begin{equation}\label{equation_dislocations_1.1}
[g_n^j]^{-1}[e^{it_n^j|\nabla|^{\alpha}}][e^{-it_n^k|\nabla|^{\alpha}}][g_n^k]f\to G^{jk}f
\end{equation}
as $n\to\infty$ in the $L_x^2$ norm topology for all Schwartz functions $f$. Note that a simple computation shows
\begin{align}
&2\pi \l|[g_n^j]^{-1}[e^{it_n^j|\nabla|^{\alpha}}][e^{-it_n^k|\nabla|^{\alpha}}][g_n^k]\phi(x)\r| \notag\\
&=[g_n^j]^{-1} (h_n^k)^{-\frac{1}{2}}\l|e^{ix\xi_n^k} \int_{\bR} e^{i\xi\frac{x-x_n^k}{h_n^k}-i|\xi+h_n^k\xi_n^k|^{\alpha}\frac{t_n^j-t_n^k}{(h_n^k)^{\alpha}}}\hat{\phi}(\xi)\ddd\xi\r| \notag\\
&=\l(\frac{h_n^j}{h_n^k}\r)^{\frac{1}{2}}\l|e^{ix h_n^j(\xi_n^k-\xi_n^j)} \int_{\bR} e^{i\xi\frac{h_n^j x+x_n^j-x_n^k}{h_n^k}-i|\xi+h_n^k\xi_n^k|^\alpha\frac{t_n^j-t_n^k}{(h_n^k)^{\alpha}}} \hat{\phi}(\xi) \ddd\xi \r|. \label{equation_dislocations_1.2}
\end{align}

We first eliminate the case $\lim_{n\to\infty}\l(h_n^j/h_n^k+h_n^k/h_n^j\r)=\infty$. Due to the fact that the operators in condition \eqref{equation_dislocations_1} are unitary operators on $L_x^2$, it is easy to conclude
\begin{align}
&2\pi \l\langle [g_n^j]^{-1}[e^{it_n^j|\nabla|^{\alpha}}][e^{-it_n^k|\nabla|^{\alpha}}][g_n^k]\phi, \psi\r\rangle \label{equation_dislocations_1.5}\\
&=(h_n^jh_n^k)^{-\frac{1}{2}}\l\langle\int_{\bR} e^{i\frac{x-x_n^j}{h_n^j}\xi+i\frac{t_n^j}{(h_n^j)^{\alpha}}|\xi|^{\alpha}} [e^{i(\cdot)h_n^j\xi_n^j}\phi]^{\wedge}(\xi)\ddd\xi, \int_{\bR} e^{i\frac{x-x_n^k}{h_n^k}\xi+i\frac{t_n^k}{(h_n^k)^{\alpha}}|\xi|^{\alpha}} [e^{i(\cdot)h_n^k\xi_n^k}\psi]^{\wedge}(\xi)\ddd\xi\r\rangle \label{equation_dislocations_1.8}\\
&=:(h_n^j h_n^k)^{-\frac{1}{2}}\l\langle\Phi_n^j\l(\frac{x-x_n^j}{h_n^j}\r), \Phi_n^k\l(\frac{x-x_n^k}{h_n^k}\r)\r\rangle. \label{equation_dislocations_1.85}
\end{align}
Notice that $\Phi_n^j\in L^2$ which implies
\begin{equation}\label{equation_dislocations_1.9}
\lim_{R\to\infty} \int_{|y|>R}|\Phi_n^j(y)|^2\ddd y=0.
\end{equation}
Hence if we setting
\[B_n^j(R):=\l\{x: \l|\frac{x-x_n^j}{h_n^j}\r|\leq R\r\}, \quad B_n^k(R):=\l\{x: \l|\frac{x-x_n^k}{h_n^k}\r|\leq R\r\}\]
and considering \eqref{equation_dislocations_1.5} with the integral on $\bR\setminus B_n^j(R)$, \holder inequality will give a bound as follows
\[2\pi \l\langle [g_n^j]^{-1}[e^{it_n^j|\nabla|^{\alpha}}][e^{-it_n^k|\nabla|^{\alpha}}][g_n^k]\phi, \psi\r\rangle_{\bR\setminus B_n^j(R)}\leq \l(\int_{|y|>R}|\Phi_n^j(y)|^2\ddd y\r)^{\frac{1}{2}}\l(\int_{\bR}|\Phi_n^k(y)|^2 \ddd y\r)^{\frac{1}{2}}.\]
Similar approach also works for the integral on $\bR\setminus B_n^k(R)$ in \eqref{equation_dislocations_1.5}. Thus, by the fact that $\Phi_n^j$ and $\Phi_n^k$ are $L_x^{\infty}$ functions, we aim to show the following estimate
\begin{equation}\label{equation_dislocations_2}
\lim_{n\to\infty} (h_n^j h_n^k)^{-\frac{1}{2}} \Big|B_n^j(R) \cap B_n^k(R)\Big|=0,
\end{equation}
which will lead to a contradiction to the assumption \eqref{equation_dislocations_1}. One observation we need is
\[\Big|B_n^j(R) \cap B_n^k(R)\Big| \leq C_R \min\l\{h_n^j, h_n^k\r\}.\]
Then we obtain the desired estimate \eqref{equation_dislocations_2} immediately since $h_n^j/h_n^k$ goes to either zero or infinity. Consequently, we can assume $h_n^j\sim h_n^k$ from now on.

Next we eliminate the case $\lim_{n\to\infty} (h_n^j+h_n^k)|\xi_n^j-\xi_n^k|=\infty$. By the Plancherel theorem and the fact that these operators are unitary operators in on $L^2(\bR)$, we conclude
\begin{align*}
\l\langle [g_n^j]^{-1}[e^{it_n^j|\nabla|^{\alpha}}][e^{-it_n^k|\nabla|^{\alpha}}][g_n^k]\phi, \psi\r\rangle_{x}
&=\l\langle [e^{-it_n^k|\nabla|^{\alpha}}][g_n^k]\phi, [e^{-it_n^j|\nabla|^{\alpha}}][g_n^j]\psi\r\rangle_{x} \\
&\sim \l\langle e^{-it_n^k|\cdot|^{\alpha}} \widehat{[g_n^k]\phi}, e^{-it_n^j|\cdot|^{\alpha}} \widehat{[g_n^j]\psi}\r\rangle_{\xi} \\
&=(h_n^jh_n^k)^{1/2} \l\langle \hat{\phi}\l(h_n^k\xi-h_n^k\xi_n^k\r), \hat{\psi}\l(h_n^j\xi-h_n^j\xi_n^j\r)\r\rangle_{\xi}.
\end{align*}
Thus the assumption $h_n^j\sim h_n^k$ gives
\begin{align*}
\l\langle [g_n^j]^{-1}[e^{it_n^j|\nabla|^{\alpha}}][e^{-it_n^k|\nabla|^{\alpha}}][g_n^k]\phi, \psi\r\rangle_{x}
&\sim\l\langle \hat{\phi}\l(\frac{h_n^k}{h_n^j}(\xi-h_n^j\xi_n^k)\r), \hat{\psi}\l(\xi-h_n^j\xi_n^j\r)\r\rangle_{\xi} \\
&\sim\l\langle \hat{\phi}\l(\xi-h_n^k\xi_n^k\r), \hat{\psi}\l(\frac{h_n^j}{h_n^k}(\xi-h_n^k\xi_n^j)\r)\r\rangle_{\xi}.
\end{align*}
Then the condition \eqref{equation_dislocations_1} together with the assumption that $\hat{\phi}$ and $\hat{\psi}$ have compact supports imply, up to subsequences, the following
\[\lim_{n\to\infty} (h_n^j+h_n^k)\l|\xi_n^k-\xi_n^j\r|=c_2, \quad c_2<\infty.\]
Hence we can assume that $(h_n^j,\xi_n^j) \equiv (h_n^k,\xi_n^k) \equiv (h_n,\xi_n)$ from now on.

With the assumptions for $\xi_n^{\cdot}$ and $h_n^{\cdot}$ at hand, we can turn the expression \eqref{equation_dislocations_1.85} into
\[\l\langle [g_n^j]^{-1}[e^{it_n^j|\nabla|^{\alpha}}][e^{-it_n^k|\nabla|^{\alpha}}][g_n^k]\phi, \psi\r\rangle= (2\pi h_n)^{-1}\l\langle\Phi_n^j\l(\frac{x-x_n^j}{h_n}\r), \Phi_n^k\l(\frac{x-x_n^k}{h_n}\r)\r\rangle.\]
Then just as what we have done above, recalling the condition \eqref{equation_dislocations_1.9}, a similar changing of variables argument and the assumption \eqref{equation_dislocations_1} imply, up to subsequences, that
\[\lim_{n\to\infty} \frac{x_n^j-x_n^k}{h_n}=c_3, \quad |c_3|<\infty.\]
On the other hand, we can turn the expression \eqref{equation_dislocations_1.8} into
\begin{equation}\label{equation_dislocations_2.3}
\l\langle [g_n^j]^{-1}[e^{it_n^j|\nabla|^{\alpha}}][e^{-it_n^k|\nabla|^{\alpha}}][g_n^k]\phi, \psi\r\rangle =(2\pi h_n)^{-1}\l\langle\tilde{\Phi}_n^j\l(\frac{x}{h_n}+\frac{t_n^j}{(h_n)^{\alpha}}\r), \tilde{\Phi}_n^k\l(\frac{x}{h_n}+\frac{t_n^k}{(h_n)^{\alpha}}\r)\r\rangle,
\end{equation}
where the function $\tilde{\Phi}_n^j$ is defined by the following
\[\tilde{\Phi}_n^j\l(\frac{x}{h_n}+\frac{t_n^j}{(h_n)^{\alpha}}\r):= e^{i\l(\frac{x}{h_n}+\frac{t_n^j}{(h_n)^{\alpha}}\r)\xi+i\frac{t_n^j}{(h_n^j)^{\alpha}}|\xi|^{\alpha}} e^{-i\l(\frac{x_n^j}{h_n}+\frac{t_n^j}{(h_n)^{\alpha}}\r)\xi} [e^{i(\cdot)h_n^j\xi_n^j}\phi]^{\wedge}(\xi)\ddd\xi,\]
and similarly for the definition of $\tilde{\Phi}_n^k$. Then we still have the fact $\tilde{\Phi}_n^j\in L^2$ and further
\[\lim_{R\to\infty} \int_{|y|>R}|\tilde{\Phi}_n^j|^2\ddd y=0.\]
Analogously, the expression \eqref{equation_dislocations_2.3} and a changing of variables argument imply, up to subsequences, that
\[\lim_{n\to\infty}\frac{t_n^j-t_n^k}{(h_n)^{\alpha}}=c_4, \quad |c_4|<\infty\]
based on the assumption \eqref{equation_dislocations_1}. Moreover, we can turn the expression \eqref{equation_dislocations_1.2} into
\[\l|[g_n^j]^{-1}[e^{it_n^j|\nabla|^{\alpha}}][e^{-it_n^k|\nabla|^{\alpha}}][g_n^k]\phi(x)\r|=\frac{1}{2\pi}\l|\int_{\bR} e^{i\Phi_n^{jk}(x,\xi)} \hat{\phi}(\xi) \ddd\xi\r|,\]
where
\begin{equation}\label{equation_dislocations_2.5}
\Phi_n^{jk}(x,\xi):=\xi\l(x+\frac{x_n^j-x_n^k}{h_n}\r)-\frac{t_n^j-t_n^k}{(h_n)^{\alpha}}|\xi+h_n\xi_n|^\alpha.
\end{equation}
It is obvious that $\int_{\bR} e^{i\Phi_n^{jk}(x,\xi)} \hat{\phi}(\xi) \ddd\xi \in L_x^{\infty}$. Next we are going to use the method of stationary phase to obtain the decay estimates of this oscillatory integral. To begin with, analysing piece by piece if necessary, we can assume $\xi+h_n\xi_n>0$ without loss of generality and rewrite $\Phi_n^{jk}(x,\xi)$ as
\begin{align}
\Phi_n^{jk}(x,\xi)&=\xi\l(x+\frac{x_n^j-x_n^k}{h_n}\r)+\sum_{m=1}^{\alpha}\frac{\binom{\alpha}{m}(t_n^k-t_n^j)(\xi_n)^{\alpha-m} (\xi)^m}{(h_n)^{m}} \label{equation_dislocations_2.7}\\
&=:\xi x+\sum_{m=1}^\alpha a_n^{m,j,k} (\xi)^m, \notag
\end{align}
where $a_n^{m,j,k}$ are the coefficients of the $m$-order term $(\xi)^m$ in the expression of $\Phi_n^{jk}$, except for the case $m=1$ where $x+a_n^{1,j,k}$ is the coefficient of $\xi$. Note that we have ignored the constant term $(\xi)^0$ here and in the computation \eqref{equation_dislocations_1.2}. This term is easy to manage due to the compactness of $\bS^{1}$, which will be shown after the definition \eqref{equation_dislocations_3}. We are going to prove that, after passing to a subsequence, each of the coefficients $a_n^{m,j,k}$ goes to some constant $c^{m,j,k}\neq\infty$ as $n$ goes to infinity. Then this result gives the desired function
\[\Phi^{jk}(x,\xi):=\xi x+\sum_{m=1}^\alpha c^{m,j,k} (\xi)^m\]
satisfying $\lim_{n\to\infty}\Phi_n^{jk}(x,\xi)=\Phi^{jk}(x,\xi)$. Thereby we get the desired operator $G^{jk}$ defined as
\begin{equation}\label{equation_dislocations_3}
G^{jk}f(x):=\frac{e^{i\theta^{jk}}}{2\pi}\int_{\bR} e^{i\Phi^{jk}(x,\xi)} \hat{f}(\xi) \ddd\xi,
\end{equation}
which satisfies the equation \eqref{equation_dislocations_1.1}. It should be pointed out that the term $e^{i\theta^{jk}}$, which we do not pay much attention to it before, comes from the parameters involved and the compactness of $\bS^1$ due to the fact $\l|e^{i\theta_n^{jk}}\r|=1$. We also remark that the lack of the term $e^{ix\cdot}$ in \eqref{equation_dislocations_3}, compared with the expression in \eqref{equation_dislocations_1.2}, comes from the assumption $\xi_n^j\equiv\xi_n^k\equiv\xi_n$.

It remains to be proved that for each $(m,j,k)$ there exists $|c^{m,j,k}|<\infty$ satisfying, after passing to a subsequence,
\[\lim_{n\to\infty}a_n^{m,j,k}=c^{m,j,k}.\]
If on the contrary for some fixed $m$ there holds $\lim_{n\to\infty}a_n^{m,j,k}=\infty$. Take
\begin{equation}\label{equation_dislocations_4}
m_0:=\max\{m: \lim_{n\to\infty}a_n^{m,j,k}=\infty\}.
\end{equation}
We break the proof into two cases $m_0=1$ and $m_0>1$. For the case $m_0=1$, we have the following limit relation
\[\Phi_n^{jk}(x,\xi)=(x+a_n^{1,j,k})\xi+\Phi_n^{1,j,k}(\xi)\to (x+\infty)\xi+\Phi^{1,j,k}(\xi)\]
as $n$ goes to infinity. Here the function $\Phi^{1,j,k}(\xi)$ whose coefficients are bounded is the limit function of $\Phi_n^{1,j,k}(\xi)$. Since the parameters $a_n^{1,j,k}$ just deduce translations for $\xi$ on the Fourier side, the assumption of compact Fourier supports property and the Plancherel theorem imply that
\[\lim_{n\to\infty}\l\langle [g_n^j]^{-1}[e^{it_n^j|\nabla|^{\alpha}}][e^{-it_n^k|\nabla|^{\alpha}}][g_n^k]\phi, \psi\r\rangle=0.\]
This is a contradiction to the condition \eqref{equation_dislocations_1}. For the case $m_0>1$, by the compactness of $\supp(\hat{\phi})$, we deduce the following estimates
\[\l|\l(\frac{\ddd}{\ddd\xi}\r)^{m_0}\Phi_n^{jk}(x,\xi)\r|\sim |a_n^{m_0,j,k}|, \quad \l|\l(\frac{\ddd}{\ddd\xi}\r)^{m_0+1}\Phi_n^{jk}(x,\xi)\r|<\infty\]
for $\xi\in\supp(\hat{\phi})$ and $n$ large enough. Therefore the classical van der Corput Lemma \cite[p. 334, Corollary]{Stein_1993} implies
\[\l\|[g_n^j]^{-1}[e^{it_n^j|\nabla|^{\alpha}}][e^{-it_n^k|\nabla|^{\alpha}}][g_n^k]\phi\r\|_{L_x^{\infty}}\lesssim_{\phi} |a_n^{m_0,j,k}|^{-\frac{1}{m_0}}\to0\]
as $n\to \infty$ and thus $[g_n^j]^{-1}[e^{it_n^j|\nabla|^{\alpha}}][e^{-it_n^k|\nabla|^{\alpha}}][g_n^k]\phi\rightharpoonup0$ in  $L_x^2$ as $n$ goes to infinity. This is a contradiction to \eqref{equation_dislocations_1} and finishes the proof of the first conclusion.

Now, we turn to the second conclusion. Actually the strategy is similar to the proof above for the first conclusion. We still can get the expression \eqref{equation_dislocations_2.5} even if $\alpha>1$ is a real number. Here we divide the proof into two parts: up to subsequences, either \begin{equation}\label{equation_dislocations_4.5}
\lim_{n\to\infty}|h_n\xi_n|\to\infty
\end{equation}
or $\lim_{n\to\infty}h_n\xi_n=c_5$ with $|c_5|<\infty$. For the latter case, indeed we have got the desired operator $G^{jk}$ satisfying \eqref{equation_dislocations_1.1} defined as follows
\[G^{jk}f(x):=\frac{1}{2\pi}\int_{\bR} e^{i(x+c_3)\xi-ic_4|\xi+c_5|^{\alpha}}\hat{f}(\xi)\ddd\xi.\]
Hence our last target is to deal with the case $h_n\xi_n\to\infty$ as $n$ goes to infinity. This time, noticing the compact Fourier supports assumption, we should change \eqref{equation_dislocations_2.7} to the following series\footnote{Recall that the binomial coefficient $\binom{\alpha}{m}:=\alpha(\alpha-1) \cdots (\alpha-m+1)/ m!$ is well-defined for $\alpha\notin \bZ$.}
\begin{align}
\Phi_n^{jk}(x,\xi)&=\xi\l(x+\frac{x_n^j-x_n^k}{h_n}\r)+\sum_{m=1}^{\infty}\frac{\binom{\alpha}{m}(t_n^k-t_n^j)(h_n\xi_n)^{\alpha-m} (\xi)^m}{(h_n)^{\alpha}} \label{equation_dislocations_5}\\
&=:\xi x+\sum_{m=1}^{\infty} a_n^{m,j,k} (\xi)^m \notag
\end{align}
for $n$ large enough, since we have the assumption \eqref{equation_dislocations_4.5} which can guarantee the uniform convergency of this series. Meanwhile, we can investigate further about the coefficients $a_n^{m,j,k}$ by using this assumption. Define $m_0$ as in \eqref{equation_dislocations_4}. Then if $m_0\geq2$, the assumption \eqref{equation_dislocations_4.5} will imply
\[\lim_{n\to\infty}a_n^{2,j,k}=\infty.\]
Again, the classical van der Corput Lemma will give the decay estimate
\[\l\|[g_n^j]^{-1}[e^{it_n^j|\nabla|^{\alpha}}][e^{-it_n^k|\nabla|^{\alpha}}][g_n^k]\phi\r\|_{L_x^{\infty}}\lesssim_{\phi} |a_n^{2,j,k}|^{-\frac{1}{2}}\to 0\]
as $n\to\infty$. If $m_0=1$ we do the same arguments as the proof for the first conclusion aforementioned. Analogously when $a_n^{1,j,k}$ and $a_n^{2,j,k}$ are both bounded, we can assume that up to subsequences
\[\lim_{n\to\infty}a_n^{1,j,k}=c^{1,j,k}, \quad \lim_{n\to\infty}a_n^{2,j,k}=c^{2,j,k}.\]
Then the desired operator, similar to the expression \eqref{equation_dislocations_3}, is given by
\[\tilde{G}^{jk}f(x):=\frac{e^{i\tilde{\theta}^{jk}}}{2\pi}\int_{\bR} e^{i\tilde{\Phi}^{jk}(x,\xi)} \hat{f}(\xi) \ddd\xi, \quad \tilde{\Phi}^{jk}(x,\xi):=(x+c^{1,j,k})\xi+c^{2,j,k}(\xi)^2.\]
Therefore we totally complete the proof of these two conclusions.
\end{proof}

\section{First-step decomposition: frequency and scaling}\label{Sec_First-step}
Usually the profile decomposition results are obtained by following two steps: first for the frequency-scaling parameters based on some refinement of Strichartz estimates which can be deduced by the bilinear restriction estimates from \cite{Tao_2003, TVV_1998, Wolff_2001}, and second for the time-space translations by using some weak convergence arguments which will be further discussed in Section \ref{Sec_Second-step} later. There may be some papers providing slightly different procedures by using similar ingredients such as \cite[Appendix A]{Tao_2009} and \cite[Theorem 4.26]{KV_2013}. We refer to \cite{Stovall_2020} for a brief discussion on the $L^2$-based linear profile decomposition and a generalization in the $L^p$ setting, see also \cite{BS_2021} for some recent results on the $L^p$-generalization.

In this section, we present the first-step decomposition by following the proofs in \cite{BOQ_2020, JPS_2010}, similar method can also be seen in some earlier papers \cite{CK_2007, KPV_2000}. It is convenient to give the following dyadic intervals in $\bR$ to do some dyadic analysis.

\begin{definition}[Dyadic intervals]
Given $j\in\bZ$, the \textit{dyadic intervals of length $2^j$} in $\bR$ is defined by
\[\mathcal{D}_j:=\l\{2^j[k,k+1): k\in\bZ\r\};\]
and we use $\mathcal{D}:=\cup_{j\in\bZ}\mathcal{D}_j$ to denote the set of all the \textit{dyadic intervals} in $\bR$.
\end{definition}

\begin{prop}[$\alpha$-refined Strichartz] \label{Prop_alpha-Refined Strichartz}
For any $p>1$, we have
\begin{equation}\label{Prop_alpha-Refined Strichartz_1}
\l\|[D^{\frac{\alpha-2}{6}}][e^{it|\nabla|^{\alpha}}]f\r\|_{L_{t,x}^6(\bR^2)} \lesssim_{\alpha, p} \l(\sup_{\tau} |\tau|^{\frac{1}{2}-\frac{1}{p}} \|\hat{f}\|_{L^p(\tau)}\r)^{\frac{1}{3}} \l\|f\r\|_{L^2(\bR)}^{\frac{2}{3}},
\end{equation}
where $\tau$ denotes an interval in $\bR$ with the length $|\tau|$. Moreover, we can restrict $\tau$ to be dyadic intervals.
\end{prop}

\begin{proof}[\textbf{Proof of Proposition \ref{Prop_alpha-Refined Strichartz}}]
We adapt the proofs in \cite[Lemma 1.2]{JPS_2010} and \cite[Section 2]{BOQ_2020} by using the Whitney decomposition and Hausdorff-Young inequality instead of bilinear restriction estimates aforementioned since we are dealing with the one dimensional case now. See also \cite{CK_2007, KPV_2000} for different methods using Fefferman-Phong's weighted inequality from \cite{Fefferman_1983}.

Notice that we can normalize $\sup_{\tau\in\mathcal{D}}|\tau|^{1/2-1/p}\|\hat{f}\|_{L^p(\tau)}=1$ for given $p>1$. This implies that the following inequality
\begin{equation}\label{equ_alpha-Refined Strichartz_0.5}
\int_{I} |\hat{f}|^p \ddd\xi\leq |I|^{1-p/2}
\end{equation}
holds for all dyadic intervals $I\in\{2^j[k,k+1): j\in\bZ, k\in\bZ\}$.  In our proof here we aim to show that
\begin{equation}\label{equ_alpha-Refined Strichartz_1}
\l\|[D^{\frac{(\alpha-2)}{6}}][e^{it|\nabla|^{\alpha}}]f [D^{\frac{(\alpha-2)}{6}}][e^{it|\nabla|^{\alpha}}]g\r\|_{L^3(\bR^2)}^{\frac{3}{2}} \lesssim \int_{\bR^2} \frac{|\hat{f}(\xi)\hat{g}(\eta)|^{\frac{3}{2}}}{|\xi-\eta|^{\frac{1}{2}}}\ddd\xi\ddd\eta.
\end{equation}
Getting the desired result \eqref{Prop_alpha-Refined Strichartz_1} from the estimates \eqref{equ_alpha-Refined Strichartz_0.5} and \eqref{equ_alpha-Refined Strichartz_1} is a standard application of Whitney decomposition. Since the details of this process can be found in \cite[Lemma 1.2]{JPS_2010} and \cite[Proposition 2.7]{BOQ_2020}, we omit the detailed proof of this part for avoiding too much repetition.

Define the, in some sense, extension operator $[E_{\alpha}]$ by
\[[E_{\alpha}]f(t,x):=2\pi [D^{\frac{\alpha-2}{6}}][e^{it|\nabla|^{\alpha}}]f(x) =\int_{\bR}e^{ix\xi-it|\xi|^{\alpha}}|\xi|^{\frac{\alpha-2}{6}} \hat{f}(\xi)\ddd\xi.\]
Then we investigate the following bilinear forms
\[[E_{\alpha}]f[E_{\alpha}]g(t,x)=\int_{\bR^2}e^{ix(\xi+\eta)-it(|\xi|^{\alpha}+|\eta|^{\alpha})}|\xi|^{\frac{\alpha-2}{6}}|\eta|^{\frac{\alpha-2}{6}} \hat{f}(\xi)\hat{g}(\eta) \ddd\xi\ddd\eta.\]
Consider the changing of variables
\begin{equation}\label{equ_alpha-Refined Strichartz_2}
(\xi, \eta)\mapsto (u,v):=(\xi+\eta, -|\xi|^{\alpha}-|\eta|^{\alpha}).
\end{equation}
Recall that for fixed $(u_0,v_0)$, the graph of the function $u_0=\xi+\eta$ is a line and the graph of $v_0=-|\xi|^{\alpha}-|\eta|^{\alpha}$ is a ``circle'' in some sense. This implies that the map defined in \eqref{equ_alpha-Refined Strichartz_2} is an at most 2-to-1 map from $\bR^2$ to the region $Q:=\{(u,v): -v\geq 2^{1-\alpha}|u|^{\alpha}\}$ which comes from the convexity. Further the Jacobian is given by
\[J(u,v)=J^{-1}(\xi, \eta)=\frac{\partial(u,v)}{\partial(\xi, \eta)}=\alpha(\xi|\xi|^{\alpha-2}-\eta|\eta|^{\alpha-2}).\]
Thus we conclude
\[\Big|[E_{\alpha}]f[E_{\alpha}]g(t,x)\Big|\leq 2\l|\int_{Q} e^{ixu+itv} |\xi\eta|^{\frac{\alpha-2}{6}} \hat{f}(\xi)\hat{g}(\eta) J^{-1}(u,v)\ddd u\ddd v\r|,\]
where $(\xi,\eta)$ is a function of $(u,v)$ via the change of variables \eqref{equ_alpha-Refined Strichartz_2} above. By the symmetry, we can assume $|\eta|\leq |\xi|$ without loss of generality. Using the Hausdorff-Young inequality and then changing variables back to $(\xi,\eta)$ we deduce the following
\begin{align}
\l\|[E_{\alpha}]f\cdot [E_{\alpha}]g\r\|^{3/2}_{L_{t,x}^3(\bR^2)} &\lesssim \l\||\xi\eta|^{\frac{\alpha-2}{6}}\hat{f}(\xi)\hat{g}(\eta) J^{-1}(u,v)\r\|^{3/2}_{L_{u,v}^{3/2}(\bR^2)} \notag\\
&=\l\||\xi\eta|^{\frac{\alpha-2}{6}} \l|J(\xi,\eta)\r|^{\frac{1}{3}} \hat{f}(\xi)\hat{g}(\eta)\r\|^{3/2}_{L_{u,v}^{3/2}(\bR^2)}. \label{equ_alpha-Refined Strichartz_3}
\end{align}
To estimate the norm above, our next target is the Jacobian factor
\[\tilde{J}(\xi,\eta):=|\xi\eta|^{\frac{\alpha-2}{4}} |J(\xi,\eta)|^{\frac{1}{2}}=\frac{|\xi\eta|^{\frac{\alpha-2}{4}}}{\l[\alpha(\xi|\xi|^{\alpha-2}-\eta|\eta|^{\alpha-2})\r]^{1/2}}.\]
If $\xi\eta\leq 0$, it is easy to see that
\[\tilde{J}(\xi,\eta)=\frac{|\xi\eta|^{\frac{\alpha-2}{4}}}{\l[\alpha(|\xi|^{\alpha-1}+|\eta|^{\alpha-1})\r]^{1/2}}\lesssim_{\alpha} (|\xi|+|\eta|)^{-\frac{1}{2}}=|\xi-\eta|^{-\frac{1}{2}}.\]
If $\xi\eta>0$ and $|\xi|\geq |\eta|$, then we have
\[|\xi|^{\alpha-1}-|\eta|^{\alpha-1}\sim_{\alpha} (|\xi|-|\eta|)|\xi|^{\alpha-2}.\]
This estimate leads to
\[\tilde{J}(\xi,\eta)=\frac{|\xi\eta|^{\frac{\alpha-2}{4}}}{\l[\alpha(|\xi|^{\alpha-1}-|\eta|^{\alpha-1})\r]^{1/2}}\lesssim_{\alpha} \frac{|\xi\eta|^{\frac{\alpha-2}{4}}}{|\xi|^{\frac{\alpha-2}{2}}|\xi-\eta|^{\frac{1}{2}}}\leq |\xi-\eta|^{-\frac{1}{2}}.\]
If $\xi\eta>0$ and $|\xi|<|\eta|$, by the symmetry, analogously as above we can obtain $\tilde{J}(\xi,\eta)\lesssim_{\alpha} |\xi-\eta|^{-\frac{1}{2}}$. In summary, we know that
\[\tilde{J}(\xi,\eta)\lesssim_{\alpha} |\xi-\eta|^{-\frac{1}{2}}\]
holds uniformly in $\xi$ and $\eta$. Taking this into the expression \eqref{equ_alpha-Refined Strichartz_3}, we get the desired estimate \eqref{equ_alpha-Refined Strichartz_1}.
\end{proof}

Based on the refined Strichartz estimate Proposition \ref{Prop_alpha-Refined Strichartz}, we can extract the frequency and scaling parameters by following a standard approach in \cite{JPS_2010}, similar argument can also be seen in \cite{CK_2007}. We omit the detailed proof of the following Lemma \ref{Lemma_frequency and scaling} here, since it is too long but essentially the same as \cite[Lemma 5.1]{JPS_2010} and \cite[Lemma 3.3]{CK_2007}.
\begin{lemma}\label{Lemma_frequency and scaling}
Let $\{u_n\}_{n\geq 1}$ be a sequence of functions with $\|u_n\|_{L_x^2(\bR)}\leq 1$. Then up to subsequences, for any $\delta>0$, there exist
\[N=N(\delta),\quad \l\{(\rho_n^{\beta},\xi_n^{\beta})_{1\leq \beta\leq N}\r\}\subset (0,\infty)\times \bR, \quad \l\{(f_n^{\beta})_{1\leq \beta \leq N} \r\}\subset L_x^2(\bR)\]
such that
\begin{equation}\label{Lemma_frequency and scaling_1}
u_n=\sum_{\beta=1}^{N} f_n^{\beta}+q_n^N
\end{equation}
and there exists a compact set $K=K(N)$ in $\bR$ such that for every $1\leq \beta\leq N$ there holds
\begin{equation}\label{Lemma_frequency and scaling_2}
(\rho_n^{\beta})^{\frac{1}{2}}\l|\hat{f}_n^\beta(\rho_n^{\beta}\xi+\xi_n^{\beta})\r|\leq C_{\delta} \mathbb{I}_K(\xi).
\end{equation}
Here the sequence $(\rho_n^{\beta},\xi_n^{\beta})$ satisfies that if $\beta\neq\gamma$ then
\begin{equation}\label{Lemma_frequency and scaling_3}
\lim_{n\to\infty}\l(\frac{\rho_n^{\beta}}{\rho_n^\gamma} +\frac{\rho_n^\gamma}{\rho_n^{\beta}} +\frac{|\xi_n^{\beta}-\xi_n^\gamma|}{\rho_n^{\beta}}+\frac{|\xi_n^{\beta}-\xi_n^\gamma|}{\rho_n^{\gamma}}\r)=\infty.
\end{equation}
The remainder term $q_n^N$ has a negligible Strichartz norm
\begin{equation}\label{Lemma_frequency and scaling_4}
\l\|[D^{\frac{\alpha-2}{6}}][e^{it|\nabla|^{\alpha}}]q_n^N\r\|_{L_{t,x}^6}\leq \delta;
\end{equation}
and furthermore, if for each $1\leq N'\leq N$ we generally define
\[q_n^{N'}:=q_n^N+f_n^N+f_n^{N-1}+\cdots+f_n^{N'+1},\]
then we have the $L^2$-almost orthogonal identity
\begin{equation}\label{Lemma_frequency and scaling_5}
\lim_{n\to\infty}\l(\|u_n\|_{L^2}^2-\l(\sum_{\beta=1}^{N'}\|f_n^{\beta}\|_{L^2}^2+\|q_n^{N'}\|_{L^2}^2\r)\r)=0.
\end{equation}
\end{lemma}

\begin{remark}\label{Remark_frequency and scaling}
We should remark that in the proof of Lemma \ref{Lemma_frequency and scaling}, by the construction, we know that the Fourier supports of $f_n^{\beta}$ and $q_n^N$ are mutually disjoint. This crucial fact also implies the conclusion \eqref{Lemma_frequency and scaling_5}. On the other hand, define operators $[\tilde{G}_n^{\beta}]$ on the Fourier side by
\[[\tilde{G}_n^\beta][\hat{f}](\xi):=(\rho_n^{\beta})^{\frac{1}{2}}\hat{f}(\rho_n^{\beta}\xi+\xi_n^{\beta}).\]
Then the conclusion \eqref{Lemma_frequency and scaling_3} means that, in view of the conditional dislocation property Proposition \ref{Prop_dislocation property}, the sequence of operators satisfy
\[[\tilde{G}_n^\beta][\tilde{G}_n^\gamma]^{-1}\rightharpoonup 0\;\; \text{and}\;\; [\tilde{G}_n^\beta]^{-1}[\tilde{G}_n^\gamma]\rightharpoonup 0\]
as $n\to\infty$ for every $\beta\neq \gamma$. Or equivalently on the spatial side, define
\[[G_n^\beta]f(x):=\mathscr{F}^{-1}[\tilde{G}_n^{\beta}]\mathscr{F}f(x) =(\rho_n^{\beta})^{-\frac{1}{2}}e^{-ix(\rho_n^{\beta})^{-1}\xi_n^{\beta}}f\l(\frac{x}{\rho_n^{\beta}}\r).\]
Then the conclusion \eqref{Lemma_frequency and scaling_3} implies that $[G_n^\beta][G_n^\gamma]^{-1}$ and $[G_n^\beta]^{-1}[G_n^\gamma]$ goes to zero as $n$ go to infinity in the weak operator topology of $\mathcal{B}(L^2)$ for $\beta\neq \gamma$. This comes from the dual approach on $L^2(\bR)$ and Plancherel theorem as follows
\begin{align*}
\l\langle [\tilde{G}_n^{\beta}][\tilde{G}_n^{\gamma}]^{-1}[\hat{f}], \hat{g}\r\rangle_{\xi}
&=\l\langle [\tilde{G}_n^{\gamma}]^{-1}\hat{f}, [\tilde{G}_n^{\beta}]^{-1}\hat{g}\r\rangle_{\xi} =\l\langle \mathscr{F}[G_n^\gamma]^{-1}f, \mathscr{F}[G_n^\beta]^{-1}g\r\rangle_{\xi} \\
&\sim \l\langle[G_n^\gamma]^{-1}f, [G_n^\beta]^{-1}g\r\rangle_x = \l\langle[G_n^\beta][G_n^\gamma]^{-1}f, g\r\rangle_x.
\end{align*}
\end{remark}

\section{Second-step decomposition: time and space translations}\label{Sec_Second-step}
After the first-step decomposition Lemma \ref{Lemma_frequency and scaling}, indeed we have obtained the desired frequency and scaling parameters. Hence in this section, we are devoted to getting the time and space translation parameters. Recall that the dislocation property (or equivalently the Galilean invariance) always play an important role in the classical case \cite{BV_2007, Bourgain_1998, CK_2007, MV_1998}. However this Galilean invariance is not valid in our $\alpha$-Strichartz setting and also note that $\alpha$ may not be a natural number. Thus our strategy is using the conditional dislocation property Proposition \ref{Prop_dislocation property} obtained in Section \ref{Sec_Dislocation from van der Corput}.

To begin this section, one ingredient we need is the following local restriction Lemma \ref{Lemma_localized restriction estimates}. Then we are ready to further decompose the functions $f_n$ obtained in the first-step decomposition and get the time-space translation parameters in Lemma \ref{Lemma_space and time translations}.
\begin{lemma}[Localized restriction]\label{Lemma_localized restriction estimates}
For $4<q<6$ and $\hat{F}\in L^{\infty}(B(\xi_0,R))$ with some $R>0$, we have
\[\l\|[D^{\frac{\alpha-2}{q}}][e^{it|\nabla|^{\alpha}}] F\r\|_{L_{t,x}^q}\leq C_{q,R} \|\hat{F}\|_{L^{\infty}(B(\xi_0, R))}.\]
\end{lemma}

\begin{proof}[\textbf{Proof of Lemma \ref{Lemma_localized restriction estimates}}]
Similarly as what we have done in the proof of Proposition \ref{Prop_alpha-Refined Strichartz}, the desired estimate is equivalent to the following bilinear form
\begin{equation}\label{equation_localized restriction estimates_1}
\l\|\int_{B(\xi_0,R)}\int_{B(\xi_0,R)} e^{ix(\xi+\eta)-it(|\xi|^{\alpha}+|\eta|^{\alpha})}|\xi|^{\frac{\alpha-2}{q}}|\eta|^{\frac{\alpha-2}{q}} \hat{F}(\xi)\hat{F}(\eta) \ddd\xi\ddd \eta\r\|_{L_{t,x}^{\frac{q}{2}}} \lesssim_{q,R} \|\hat{F}\|^2_{L^{\infty}(B(\xi_0,R))}.
\end{equation}
By changing of variables
\[(u,v):=(\xi+\eta, -|\xi|^{\alpha}-|\eta|^{\alpha}),\]
using the Hausdorff-Young inequality and then changing the variables back, we conclude that the left hand side of \eqref{equation_localized restriction estimates_1} is bounded by
\begin{equation}\label{equation_localized restriction estimates_2}
C \l(\int_{B(\xi_0, R)\times B(\xi_0, R)} |\hat{F}(\xi)\hat{F}(\eta)|^{r'} |\xi|^{\frac{(\alpha-2)r'}{2r}}|\eta|^{\frac{(\alpha-2)r'}{2r}} |J(\xi,\eta)|^{r'-1} \ddd\xi\ddd\eta\r)^{\frac{1}{r'}}
\end{equation}
where
\[r:=\frac{q}{2}\in(2,3),\quad J(\xi,\eta)^{-1}:=\alpha(\xi|\xi|^{\alpha-2}-\eta|\eta|^{\alpha-2}).\]
We then consider the Jacobian factor
\[\tilde{J}(\xi,\eta):=|\xi\eta|^{\frac{(\alpha-2)r'}{2r}}|J(\xi,\eta)|^{r'-1}= \frac{|\xi\eta|^{\frac{(\alpha-2)(r'-1)}{2}}}{\big|\xi|\xi|^{\alpha-2}-\eta|\eta|^{\alpha-2}\big|^{r'-1}}.\]
It is easy to see that $\tilde{J}$ can only has singularity at the following singular line
\[\xi=\eta.\]
By investigating the order of the singularity of $\tilde{J}$ at this singular line, we know that
\[\l\|\tilde{J}(\xi,\eta)\r\|_{L^1_{\xi,\eta}(B(\xi_0,R)\times B(\xi_0,R))}\lesssim_{R,r'} 1.\]
Therefore we can control \eqref{equation_localized restriction estimates_2} by
\[C_{q,R} \|\hat{F}\|^2_{L^{\infty}(B(\xi_0,R))},\]
which leads to the desired result \eqref{equation_localized restriction estimates_1} and thereby the proof is completed.
\end{proof}

\begin{definition}[Limit-orthogonality for sequences of operators]
For fixed $j\neq k$, we say that two sequences of operators $([g_n^j])$ and $([g_n^k])$ in $\mathcal{B}(L^2)$ are \textit{limit-orthogonal} if
\[[g_n^j]^{-1}[g_n^k]\rightharpoonup 0, \quad n\to\infty.\]
\end{definition}

\begin{lemma}[Time-space translations]\label{Lemma_space and time translations}
Let $\mathbb{F}:=(f_n)_{n\geq1}$ be a sequence of $L^2(\bR)$ functions. Define the unitary operators $[\tilde{G}_n]$ and $[G_n]$ on $L^2(\bR)$ by
\[[\tilde{G}_n]f(x):=(\rho_n)^{\frac{1}{2}}f(\rho_nx+\xi_n),\quad [G_n]f(x):=\mathscr{F}^{-1}[\tilde{G}_n]\mathscr{F}f(x)=(\rho_n)^{-\frac{1}{2}} e^{-ix\frac{\xi_n}{\rho_n}} f\l(\frac{x}{\rho_n}\r).\]
If we assume that the following condition
\[\l|[\tilde{G}_n][\hat{f_n}](\xi)\r|\leq \hat{F}(\xi),\quad \hat{F}\in L^{\infty}(K)\]
holds for some compact set $K\subset \bR$ independent of $n$. Then up to subsequences, there exist
\[\{(s_n^j, y_n^j)_{j\geq1}\}\subset \bR\times\bR,\quad \{(\phi^j)_{j\geq1}\}\subset L^2(\bR),\quad [g_n^j]\phi(x):=[e^{-is_n^j|\nabla|^{\alpha}}]\phi(x-y_n^j)\]
such that the operators $[g_n^j][G_n]^{-1}$ satisfy the following limit-orthogonality property
\begin{equation}\label{Lemma_space and time translations_1}
[G_n][g_n^j]^{-1}[g_n^k][G_n]^{-1}\rightharpoonup 0, \quad n\to\infty
\end{equation}
for every $j\neq k$. Meanwhile, for every $M\geq1$ there exist $e_n^M\in L^2(\bR)$ and the decomposition
\begin{equation}\label{Lemma_space and time translations_2}
f_n(x)=\sum_{j=1}^{M} [g_n^j][G_n]^{-1}\phi^j(x)+e_n^M(x)
\end{equation}
with the vanishing Strichartz norm estimate for the remainder
\begin{equation}\label{Lemma_space and time translations_3}
\lim_{M\to\infty}\lim_{n\to\infty} \l\|[D^{\frac{\alpha-2}{6}}][e^{it|\nabla|^{\alpha}}]e_n^M\r\|_{L_{t,x}^6}=0.
\end{equation}
Furthermore, for every $M\geq1$ we have the $L^2$-almost orthogonal identity
\begin{equation}\label{Lemma_space and time translations_4}
\lim_{n\to\infty}\l(\|f_n\|_{L^2}^2-\l(\sum_{j=1}^{M}\|\phi^j\|_{L^2}^2+\|e_n^M\|_{L^2}^2\r)\r)=0.
\end{equation}
\end{lemma}

\begin{proof}[\textbf{Proof of Lemma \ref{Lemma_space and time translations}}]
We adopt some ideas from \cite[Lemma 5.2]{JPS_2010} and \cite[Section 3]{CK_2007}, while similar approaches also arise in earlier papers \cite{BG_1999, Keraani_2001} and some of the references aforementioned. However as we have stated before, to generalize these classical arguments into our $\alpha$-Strichartz setting, we should use the conditional dislocation property Proposition \ref{Prop_dislocation property}. Take $\mathbb{P}:=(P_n)_{n\geq1}$ with
\[\hat{P}_n(\xi)=[\tilde{G}_n][\hat{f}_n](\xi).\]
Let $\mathscr{W}(\mathbb{P})$ be the set of weak limits in $L^2(\bR)$ for subsequences of $[G_n][g_n]^{-1}[G_n]^{-1}\mathbb{P}$ defined by
\[\mathscr{W}(\mathbb{P}):=\l\{w\!\!-\!\!\!\lim_{n\to\infty} [G_n][g_n]^{-1}[G_n]^{-1}P_n(x): (s_n,y_n)\in\bR^2\r\}, \quad [g_n]\phi(x):=[e^{-is_n|\nabla|^{\alpha}}]\phi(x-y_n),\]
and then define
\[\mu(\mathbb{P}):=\sup\l\{\|\phi\|_{L^2}: \phi\in\mathscr{W}(\mathbb{P})\r\}.\]
To get the desired decomposition \eqref{Lemma_space and time translations_2}, our strategy is to get the decomposition for $P_n$ as follows
\begin{equation}\label{equ_space and time translations_1}
P_n(x)=\sum_{j=1}^{M} [G_n][g_n^j][G_n]^{-1}\phi^j(x)+p_n^M(x),
\end{equation}
and then set $e_n^M(x):=[G_n]^{-1} p_n^M(x)=\sqrt{\rho_n} e^{ix\xi_n}p_n^M(\rho_n x)$. Similarly define the following
\[\mathbb{P}^M:=(p_n^M)_{n\geq1},\quad \mathbb{E}^M:=(e_n^M)_{n\geq1}.\]
Firstly, we claim that if the conclusion \eqref{Lemma_space and time translations_3} in Lemma \ref{Lemma_space and time translations} is replaced by
\begin{equation}\label{equ_space and time translations_2}
\lim_{M\to\infty}\mu(\mathbb{E}^M)=\lim_{n\to\infty}\mu(\mathbb{P}^M)=0,
\end{equation}
then this lemma is true even if we do not have the assumption that $K$ is a compact set independent of $n$. We show this claim as follows.

Indeed, if $\mu({\mathbb{P}})=0$, then we can take $\phi^j=0$ for all $j$ and the claim is proved. Otherwise if $\mu(\mathbb{P})>0$, we take $\phi^1\in \mathscr{W}(\mathbb{P})$ such that
\[\|\phi^1\|_{L^2}\geq \frac{\mu(\mathbb{P})}{2}>0.\]
By the definition of $\mathscr{W}(\mathbb{P})$, there exists a sequence $(s_n^1,y_n^1)\in\bR^2$ such that, up to extracting a subsequence, we have
\begin{equation}\label{equ_space and time translations_3}
[G_n][g_n^1]^{-1}[G_n]^{-1}P_n\rightharpoonup \phi^1
\end{equation}
in $L^2(\bR)$ as $n$ goes to infinity. Setting $p_n^1:=P_n-[G_n][g_n^1][G_n]^{-1}\phi^1$, then we obtain
\[\lim_{n\to\infty}\l(\|P_n\|_{L^2}^2-\|\phi^1\|_{L^2}^2-\|p_n^1\|_{L^2}^2\r)=0,\]
due to the weak convergency \eqref{equ_space and time translations_3} and the fact that $L^2(\bR)$ is a Hilbert space. Notice that all these operators involved are unitary operators on $L^2$. Therefore the almost orthogonal identity \eqref{Lemma_space and time translations_4} holds for $M=1$. Next, we replace $P_n$ by $p_n^1$ and then do this process again. If $\mu(\mathbb{P}^1)>0$, we get the function $\phi^2$, the sequence of parameters $(s_n^2, y_n^2)$ and the sequence of functions $\mathbb{P}^2$. Moreover, we have one more conclusion as follows: the sequence of operators
\[[G_n][g_n^2]^{-1}[g_n^1][G_n]^{-1}\rightharpoonup 0\]
in $\mathcal{B}(L^2)$ as $n$ goes to infinity. Indeed if this conclusion is not true, then the dislocation property Proposition \ref{Prop_dislocation property} asserts that, up to subsequences, there exists an isometric $[g^{1,2}]$ on $L^2(\bR)$ satisfying
\[[G_n]^{-1}[g_n^2]^{-1}[g_n^1][G_n]\to [g^{1,2}]\]
in $\mathcal{B}(L^2)$ as $n$ goes to infinity. Therefore the following relation
\[[G_n][g_n^2]^{-1}[G_n]^{-1}p_n^1=\l([G_n][g_n^2]^{-1}[g_n^1][G_n]^{-1}\r) [G_n][g_n^1]^{-1}[G_n]^{-1}p_n^1\]
and the weak convergency fact \eqref{equ_space and time translations_3} imply that $\phi^2=0$, which means $\mu(\mathbb{P}^2)=0$. This is a contradiction. Iterating this process leads to
\begin{equation*}\begin{array}{ccc}
p_n^1:= P_n -[G_n][g_n^1][G_n]^{-1}\phi^1, & [G_n][g_n^1]^{-1}[G_n]^{-1}P_n\rightharpoonup \phi^1, & \|\phi^1\|_{L^2}\geq \frac{\mu(\mathbb{P})}{2}>0; \\\\
p_n^2:= p_n^1 -[G_n][g_n^2][G_n]^{-1}\phi^2, & [G_n][g_n^2]^{-1}[G_n]^{-1}p_n^1\rightharpoonup \phi^2, & \|\phi^2\|_{L^2}\geq \frac{\mu(\mathbb{P}^1)}{2}>0; \\\\
\vdots &\vdots & \vdots\\\\
p_n^j:=p_n^{j-1}-[G_n][g_n^{j}][G_n]^{-1}\phi^{j}, & [G_n][g_n^j]^{-1}[G_n]^{-1}p_n^{j-1}\rightharpoonup \phi^j, & \|\phi^j\|_{L^2}\geq \frac{\mu(\mathbb{P}^{j-1})}{2}>0; \\\\
\vdots &\vdots & \vdots\\\\
\end{array}
\end{equation*}
A diagonal process yields a sequence of functions $(\phi^j)_{j\geq1}$ and a family of operators $[g_n^j]$ satisfying the orthogonal conclusion \eqref{Lemma_space and time translations_1} for the case $j=k+1$. By the construction, we get the decomposition identity \eqref{equ_space and time translations_1} and the almost orthogonal identities \eqref{Lemma_space and time translations_4}. To prove the desired claim, it remains for us to show the conclusion \eqref{Lemma_space and time translations_1} for all $j\neq k$ and the estimate \eqref{equ_space and time translations_2}. We show the estimate \eqref{equ_space and time translations_2} first. Recall that $\|f_n\|_{L^2}$ is uniformly bounded. Then \eqref{Lemma_space and time translations_4} implies
\[\sum_{j=1}^{M}\|\phi^j\|_{L^2}^2\leq\limsup_{n\to\infty}\|f_n\|_{L^2}^2\leq C.\]
Hence we know that the positive series $\sum_j \|\phi^j\|_{L^2}^2$ is convergent and further $\lim_{n\to\infty}\|\phi^j\|_{L^2}=0$. On the other hand, by the construction we have
\[\mu(\mathbb{P}^M)\leq 2\|\phi^{M+1}\|_{L^2},\]
which gives the desired estimate \eqref{equ_space and time translations_2}. Now we turn to the conclusion \eqref{Lemma_space and time translations_1}. Indeed, the more general case $j=k+m (m\in\bZ_{+})$ comes from the basic case $j=k+1$, the following identity
\[p_n^{k+m-1}=p_n^k-[G_n][g_n^{k+1}][G_n]^{-1}\phi^{k+1}- \cdots- [G_n][g_n^{k+m-1}][G_n]^{-1}\phi^{k+m-1},\]
and an inductive argument. For the case $j=k+m (m\in\bZ_{-})$, if there does not hold the following
\[[G_n][g_n^j]^{-1}[g_n^k][G_n]^{-1}\rightharpoonup 0\]
in $\mathcal{B}(L^2)$ as $n$ goes to infinity, then by the dislocation property Proposition \ref{Prop_dislocation property} we can assume
\[[G_n][g_n^j]^{-1}[g_n^k][G_n]^{-1}\to [g^{j,k}], \quad [g^{j,k}]\in\mathcal{B}(L^2), \quad [g^{j,k}]\neq 0\]
in $\mathcal{B}(L^2)$ as $n$ goes to infinity. In this case we obviously have $[g^{j,k}]^{-1}=[g^{k,j}]$. Hence we can investigate the sequence $[G_n][g_n^k]^{-1}[g_n^j][G_n]^{-1}$ and turn the case $j=k+m (m\in\bZ_{-})$ into the case $j=k+m (m\in\bZ_{+})$ which we have already proved. Therefore we complete the proof of the claim.

To totally finish the proof of this Lemma \ref{Lemma_space and time translations}, our next target is to get the desired conclusion \eqref{Lemma_space and time translations_3} from the estimate \eqref{equ_space and time translations_2} by using the localized restriction estimate Lemma \ref{Lemma_localized restriction estimates}.

Notice that we have the compact set $K$ and the operators $[g_n^j]$ do not change the support on the Fourier side. It means that when we get the above decomposition with conclusion \eqref{equ_space and time translations_2}, on the Fourier side, all the processes are taken place on this compact set $K$. Therefore, we conclude $\hat{\phi}^j\in L^{\infty}(K)$ and further $\hat{p}_n^M\in L^{\infty}(K)$. Since the Fourier support for $e_n^M$ is not ideal, we first use some scaling skills as follows
\begin{align*}
\l\|[D^{\frac{\alpha-2}{6}}][e^{it|\nabla|^{\alpha}}]e_n^M\r\|_{L_{t,x}^6}
&=\l\|[D^{\frac{\alpha-2}{6}}] [e^{it|\nabla|^{\alpha}}] [\sqrt{\rho_n} e^{i(\cdot)\xi_n}p_n^M(\rho_n \cdot)]\r\|_{L_{t,x}^6} \\
&=\l\|[D^{\frac{\alpha-2}{6}}] [e^{it|\nabla|^{\alpha}}] [e^{i(\cdot)\frac{\xi_n}{\rho_n}}p_n^M]\r\|_{L_{t,x}^6}.
\end{align*}
Then we investigate the function
\[\omega_n^M(x):=e^{ix\frac{\xi_n}{\rho_n}}p_n^M(x),\]
with the Fourier support information $\supp (\hat{\omega}_n^M)\subset K+(\rho_n)^{-1}\xi_n$. The \holder inequality and the Bernstein inequality imply that
\[\l\|[D^{\frac{\alpha-2}{6}}][e^{it|\nabla|^{\alpha}}]\omega_n^M\r\|_{L_{t,x}^6}\lesssim_K \l\|[D^{\frac{\alpha-2}{q}}][e^{it|\nabla|^{\alpha}}]\omega_n^M\r\|_{L_{t,x}^q}^{q/6} \l\|[e^{it|\nabla|^{\alpha}}]\omega_n^M\r\|_{L_{t,x}^{\infty}}^{1-q/6}\]
for $4<q<6$. Meanwhile, Lemma \ref{Lemma_localized restriction estimates} gives the following estimate
\[\l\|[D^{\frac{\alpha-2}{q}}][e^{it|\nabla|^{\alpha}}]\omega_n^M\r\|_{L_{t,x}^q} \lesssim_K 1,\]
which is independent of $n$ and $M$. Hence, to get the desired result \eqref{Lemma_space and time translations_3}, it suffices to prove
\[\lim_{M\to\infty}\limsup_{n\to\infty}\l\|[e^{it|\nabla|^{\alpha}}]\omega_n^M\r\|_{L_{t,x}^{\infty}}=0.\]
Moreover by \eqref{equ_space and time translations_2}, it suffices to prove the following claim
\begin{equation}\label{equ_space and time translations_4}
\limsup_{n\to\infty} \l\|[e^{it|\nabla|^{\alpha}}]\omega_n^M\r\|_{L_{t,x}^{\infty}}\lesssim_{K} \mu(\mathbb{E}^M).
\end{equation}
Indeed, choose an even function $\mathbb{I}_K\in C_c^{\infty}(\bR)$ satisfying $\mathbb{I}_K=1$ on $K$ and choose $(a_n, b_n)$ such that
\[\l\|[e^{it|\nabla|^{\alpha}}]\omega_n^M\r\|_{L_{t,x}^{\infty}}=\l|[e^{ia_n|\nabla|^{\alpha}}]\omega_n^M(b_n)\r|.\]
Define
\[\mathbb{I}_{K_n}(x):=\mathbb{I}_{K}\l(x-(\rho_n)^{-1}\xi_n\r), \quad \Omega_n^M(t,x):=[e^{it|\nabla|^{\alpha}}]\omega_n^M(x).\]
It follows that
\[\Omega_n^M(t,x)=\mathscr{F}^{-1}e^{-it|\xi|^{\alpha}}\mathbb{I}_{K_n}(\xi)\mathscr{F}\omega_n^M(x), \quad \Omega_n^M(a_n, x+b_n)\in \mathscr{W}(\mathbb{P}^M).\]
Then using some basic properties for the spatial Fourier transform $\mathscr{F}$ and $\mathbb{I}_K$, by \holder inequality we can control the $\l\|\Omega_n^M\r\|_{L_{t,x}^{\infty}}$ term as follows
\begin{align*}
\l\|\Omega_n^M\r\|_{L_{t,x}^{\infty}} &=\l|\Omega_n^M(a_n,b_n)\r| =\l|\mathscr{F}^{-1}[\mathbb{I}_{K_n}\hat{\Omega}_n^M](a_n,b_n)\r| \\
&\sim \lim_{n\to\infty}\l|\check{\mathbb{I}}_{K_n}\ast \Omega_n^M(a_n, b_n)\r| \\
&=\lim_{n\to\infty}\l|\int_{\bR} \check{\mathbb{I}}_{K_n}(x) \Omega_n^M(a_n, x+b_n) \ddd x\r| \\
&\leq \l\|\check{\mathbb{I}}_{K_n}\r\|_{L_x^2} \mu(\mathbb{P}^M)\lesssim_K \mu(\mathbb{P}^M).
\end{align*}
Therefore we can obtain the desired result \eqref{equ_space and time translations_4} and finish the proof.
\end{proof}

\begin{remark}\label{Remark_space and time translations}
As has been pointed out in \cite[Remark 5.3]{JPS_2010}, we can make a reduction in Lemma \ref{Lemma_space and time translations} when
\[\lim_{n\to\infty}(\rho_n)^{-1}\xi_n=a, \quad |a|<\infty.\]
In this case we can assume $\xi_n\equiv0$ since we can replace $e^{ix(\rho_n)^{-1}\xi_n}\phi^{j}$ by $e^{ixa}\phi^j$, put the difference into the error term and then regard $e^{ixa}\phi^j$ as the new $\phi^j$.
\end{remark}

\section{Profile decomposition of alpha-Strichartz version}\label{Sec_Profile decomposition}
In this section, with the two steps of decomposition Lemma \ref{Lemma_frequency and scaling} and Lemma \ref{Lemma_space and time translations} at hand, we are able to show the desired $\alpha$-Strichartz version profile decomposition results Proposition \ref{Prop_linear profile decomposition} and the Strichartz-orthogonality of profiles Proposition \ref{Prop_Strichartz-orthogonal profiles}. It should be pointed out that, in the proof of Proposition \ref{Prop_Strichartz-orthogonal profiles}, we use the conditional dislocation property Proposition \ref{Prop_dislocation property} once more to coordinate the limit-orthogonal property conclusion \eqref{Prop_linear profile decomposition_3} in Proposition \ref{Prop_linear profile decomposition}.

\begin{proof}[\textbf{Proof of Proposition \ref{Prop_linear profile decomposition}}]
Using the Lemma \ref{Lemma_frequency and scaling} with $\frac{\delta}{2}$ and then using Lemma \ref{Lemma_space and time translations} properly, we can obtain the decomposition
\begin{equation}\label{equation_linear profile decomposition_1}
u_n(x)=\sum_{\beta=1}^N\l(\sum_{j=1}^{M_{\beta}} [g_n^{\beta,j}][G_n^{\beta}]^{-1}\phi^{\beta, j}(x)\r)+ e_{n}^{N, M_1,\ldots, M_{N}}(x),
\end{equation}
where the remainder term is
\[e_{n}^{N, M_1,\ldots, M_{N}}(x):=\sum_{\beta=1}^{N}e_n^{M_{\beta}}+q_n^N\]
and the operators in \eqref{equation_linear profile decomposition_1} are defined by
\[\quad [G_n^\beta]f(x):= (\rho_n^{\beta})^{-\frac{1}{2}} e^{-ix(\rho_n^{\beta})^{-1}\xi_n^{\beta}}f\l(\frac{x}{\rho_n^{\beta}}\r), \quad [g_n^{\beta,j}]f(x):=[e^{-is_n^{\beta,j}|\nabla|^{\alpha}}]f(x-y_n^{\beta,j}).\]
Here, for each $1\leq\beta\leq N$, we choose $M_{\beta}$ to guarantee that for all $M\geq M_{\beta}$ there holds
\[\limsup_{n\to\infty} \l\|[D^{\frac{\beta-2}{6}}][e^{it|\nabla|^{\alpha}}]e_n^{M_{\beta}}\r\|_{L_{t,x}^6}\leq \frac{\delta}{2N}.\]
This is realizable since we have the vanishing Strichartz norm estimate \eqref{Lemma_space and time translations_3} for the remainder in Lemma \ref{Lemma_space and time translations}. Therefore, by combining the negligible Strichartz norm estimate \eqref{Lemma_frequency and scaling_4} for the remainder in Lemma \ref{Lemma_frequency and scaling} with $\frac{\delta}{2}$, we have the following norm estimate for the remainder term
\begin{equation}\label{equation_linear profile decomposition_2}
\limsup_{n\to\infty} \l\|[D^{\frac{\beta-2}{6}}][e^{it|\nabla|^{\alpha}}]e_{n}^{N, M_1,\ldots, M_{N}}\r\|_{L_{t,x}^6}\leq \delta.
\end{equation}

As for the sequence of operators $[G_n^{\gamma}][g_n^{\gamma,k}]^{-1}[g_n^{\beta,j}][G_n^{\beta}]^{-1}$, we will investigate the limit of this sequence in the weak operator topology of $\mathcal{B}(L^2)$ as $n$ goes to infinity. If $\beta=\gamma$, then the limit-orthogonal conclusion \eqref{Lemma_space and time translations_1} in Lemma \ref{Lemma_space and time translations} implies
\begin{equation}\label{equation_linear profile decomposition_5}
[G_n^{\gamma}][g_n^{\gamma,k}]^{-1}[g_n^{\beta,j}][G_n^{\beta}]^{-1}\rightharpoonup0
\end{equation}
in $\mathcal{B}(L^2)$ as $n$ goes to infinity. If $\beta\neq \gamma$, then the dual approach and Plancherel theorem give
\begin{align}
\lim_{n\to\infty} \l\langle[G_n^{\gamma}][g_n^{\gamma,k}]^{-1}[g_n^{\beta,j}][G_n^{\beta}]^{-1}f,g\r\rangle_x &=\lim_{n\to\infty} \l\langle[g_n^{\beta,j}][G_n^{\beta}]^{-1}f, [g_n^{\gamma, k}][G_n^{\gamma}]^{-1}g\r\rangle_x \notag\\
&\sim \lim_{n\to\infty} \l\langle\mathscr{F}[g_n^{\beta,j}][G_n^{\beta}]^{-1}f, \mathscr{F}[g_n^{\gamma, k}][G_n^{\gamma}]^{-1}g \r\rangle_{\xi}, \label{equation_linear profile decomposition_3}
\end{align}
where $f$ and $g$ can be assumed to be Schwartz functions with compact Fourier supports. Note that the operators $[g_n^{\beta,j}]$ and $[g_n^{\gamma,k}]$ do not change the Fourier supports. Hence the conclusion \eqref{Lemma_frequency and scaling_3} for the frequency and scaling parameters in Lemma \ref{Lemma_frequency and scaling}, recalling the Remark \ref{Remark_frequency and scaling}, implies the limit value in \eqref{equation_linear profile decomposition_3} is zero and further
\begin{equation}\label{equation_linear profile decomposition_4}
[G_n^{\gamma}][g_n^{\gamma,k}]^{-1}[g_n^{\beta,j}][G_n^{\beta}]^{-1} \rightharpoonup0
\end{equation}
in $\mathcal{B}(L^2)$ as $n$ goes to infinity.

For the $L^2$-orthogonality, combing the $L^2$-almost orthogonal identities \eqref{Lemma_frequency and scaling_5} and \eqref{Lemma_space and time translations_4}, we conclude
\[\lim_{n\to\infty}\l(\|u_n\|_{L^2}^2-\sum_{\beta=1}^N\l(\sum_{j=1}^{M_{\beta}} |\phi^{\beta, j}\|_{L^2}^2+\|e_n^{M_{\beta}}\|_{L^2}^2\r)-\|q_n^N\|_{L^2}^2\r)=0.\]
Recall that the Fourier supports of $q_n^N$ and $e_n^{M_{\beta}}$ are mutually disjoint which comes from the fact that the operators $[g_n^{\beta, j}]$ do not change the Fourier support and the Remark \ref{Remark_frequency and scaling}. Therefore we conclude
\begin{equation}\label{equation_linear profile decomposition_6}
\lim_{n\to\infty}\l(\|u_n\|_{L^2}^2-\l(\sum_{\beta=1}^N\sum_{j=1}^{M_{\beta}} \l\|\phi^{\beta, j}\r\|_{L^2}^2\r)-\l\|e_n^{N, M_1, \ldots, M_{\beta}}\r\|_{L^2}^2\r)=0.
\end{equation}
Notice that the parameters $N$ and $M_{\beta}$ depend only on $\delta$. Hence by enumerating the pairs $(\beta,j)$
\begin{equation}\label{equation_linear profile decomposition_7}
\{(\beta,j)<(\gamma,k)\}:=\{\beta+j<\gamma+k \;\;\text{or,}\;\; \beta+j=\gamma+k \;\;\text{and}\;\; \beta<\gamma\},
\end{equation}
and relabeling the pairs $(\beta,j)$, we can define
\[(h_n^{\tilde{j}}, x_n^{\tilde{j}}, \xi_n^{\tilde{j}}, t_n^{\tilde{j}}):=\l(1/\rho_n^{\beta}, -y_n^{\beta,j}, \xi_n^{\beta}, s_n^{\beta,j}\r), \quad [T(h_n^{\tilde{j}}, x_n^{\tilde{j}}, \xi_n^{\tilde{j}}, t_n^{\tilde{j}})]:=[g_n^{\beta,j}][G_n^{\beta}]^{-1}, \quad \omega_n^{J_{\delta}}:=e_{n}^{N, M_1,\ldots, M_{N}}.\]
Then, after a classical diagonal process, we obtain the desired decomposition \eqref{Prop_linear profile decomposition_1} by \eqref{equation_linear profile decomposition_1}; the limit-orthogonality conclusion \eqref{Prop_linear profile decomposition_3} comes from the weak operator topology convergency \eqref{equation_linear profile decomposition_4} and \eqref{equation_linear profile decomposition_5}; meanwhile the $L^2$-almost orthogonal identity \eqref{Prop_linear profile decomposition_4} comes from \eqref{equation_linear profile decomposition_6}. Therefore it remains for us to prove the Strichartz norm estimate \eqref{Prop_linear profile decomposition_2} for the remainder term in view of the enumeration \eqref{equation_linear profile decomposition_7}. Actually this is a standard consequence. By using the Strichartz-orthogonality of profiles Proposition \ref{Prop_Strichartz-orthogonal profiles}, a $3\varepsilon$ trick will give this desired result \eqref{Prop_linear profile decomposition_2}. We omit the details here for simplicity and the reader can find similar proofs in \cite[p. 107]{Shao_2009} and \cite[p. 371]{Keraani_2001}.
\end{proof}

In order to investigate the extremal problem for the $\alpha$-Strichartz estimates and complete the proof for Proposition \ref{Prop_linear profile decomposition}, we need to show the Strichartz-orthogonality of profiles. Indeed, the limit-orthogonality property \eqref{Prop_linear profile decomposition_3} for $[T_n^j]$ in Proposition \ref{Prop_linear profile decomposition} and the conditional dislocation property Proposition \ref{Prop_dislocation property} imply this desired conclusion.
\begin{proof}[\textbf{Proof of Proposition \ref{Prop_Strichartz-orthogonal profiles}}]
Without loss of generality, we focus on the Schwartz functions $\phi^j$ and $\phi^k$ whose Fourier supports are compact. Based on the conclusion \eqref{Lemma_frequency and scaling_3} for the frequency and scaling parameters in Lemma \ref{Lemma_frequency and scaling}, we first deal with the case
\[\lim_{n\to\infty} \frac{h_n^j}{h_n^k}+\frac{h_n^k}{h_n^j}=\infty.\]
Direct computation gives the following
\begin{align}
\Phi_n^{j,k}(t,x):&=[D^{\frac{\alpha-2}{6}}] [e^{it|\nabla|^{\alpha}}][T_n^j]\phi^j(x) \cdot [D^{\frac{\alpha-2}{6}}] [e^{it|\nabla|^{\alpha}}][T_n^k]\phi^k(x)\notag\\
&= (h_n^j)^{-\frac{1}{2}-\frac{\alpha-2}{6}} \frac{e^{-ix_n^j\xi_n^j}}{2\pi} \int_{\bR} |\xi|^{\frac{\alpha-2}{6}} e^{i\frac{x-x_n^j}{h_n^j}\xi-i\frac{t-t_n^j}{(h_n^j)^{\alpha}}|\xi|^{\alpha}} [e^{i(\cdot)h_n^j\xi_n^j}\phi^j]^{\wedge}(\xi)\ddd\xi \notag\\
&\quad \times (h_n^k)^{-\frac{1}{2}-\frac{\alpha-2}{6}} \frac{e^{-ix_n^k\xi_n^k}}{2\pi} \int_{\bR} |\xi|^{\frac{\alpha-2}{6}} e^{i\frac{x-x_n^k}{h_n^k}\xi-i\frac{t-t_n^k}{(h_n^k)^{\alpha}}|\xi|^{\alpha}} [e^{i(\cdot)h_n^k\xi_n^k}\phi^k]^{\wedge}(\xi)\ddd\xi \label{equation_profiles orthogonal_0.3}\\
&=:(h_n^j)^{-\frac{\alpha+1}{6}}\Phi_n^j\l(\frac{t-t_n^j}{(h_n^j)^{\alpha}}, \frac{x-x_n^j}{h_n^j}\r) \cdot (h_n^k)^{-\frac{\alpha+1}{6}}\Phi_n^k\l(\frac{t-t_n^k}{(h_n^k)^{\alpha}}, \frac{x-x_n^k}{h_n^k}\r) \notag
\end{align}
By the $\alpha$-Strichartz estimate \eqref{equation_symmetric alpha-Strichartz} we have
\[\lim_{R\to\infty}\int_{\{|s|+|y|> R\}} |\Phi_n^j(s,y)|^{6} \ddd y\ddd s=0.\]
Therefore setting
\[B_n^j(R):=\l\{(t,x): \l|\frac{t-t_n^j}{(h_n^j)^{\alpha}}\r|+\l|\frac{x-x_n^j}{h_n^j}\r|\leq R\r\},\]
\holder inequality gives
\[\int_{\bR\setminus B_n^j(R)} |\Phi_n^{j,k}|^{3} \ddd x\ddd t \leq \l(\int_{\{|x|+|t|> R\}} |\Phi_n^j|^{6} \ddd x\ddd t\r)^{\frac{1}{2}} \l(\int_{\bR} |\Phi_n^k|^{6} \ddd x \ddd t\r)^{\frac{1}{2}},\]
and analogously for $\bR\setminus B_n^k(R)$. Thus we are reduced to proving
\begin{equation}\label{equation_profiles orthogonal_1}
\lim_{n\to\infty} (h_n^j h_n^k)^{-\frac{1+\alpha}{2}} \Big|B_n^j(R) \cap B_n^k(R)\Big|=0
\end{equation}
due to the fact that $\Phi_n^j$ and $\Phi_n^k$ are $L_{t,x}^{\infty}$ functions. By the observation
\[\Big|B_n^j(R) \cap B_n^k(R)\Big| \leq C_R \min\l\{(h_n^j)^{1+\alpha}, (h_n^k)^{1+\alpha}\r\},\]
the desired estimate \eqref{equation_profiles orthogonal_1} follows immediately since $h_n^j/h_n^k$ goes to either zero or infinity. Hence we can assume $h_n^j\sim h_n^k$ from now on. Then we turn to investigate the case
\[\lim_{n\to\infty}\l(h_n^j+h_n^k\r)|\xi_n^j-\xi_n^k|=\infty.\]
By symmetry we may assume $\lim_{n\to\infty}h_n^j|\xi_n^j-\xi_n^k|=\infty$, thus from the expression \eqref{equation_profiles orthogonal_0.3} we conclude
\begin{align}
\Phi_n^{j,k}(t,x)&= (h_n^j)^{-\frac{\alpha+1}{6}} \frac{e^{-ix_n^j\xi_n^j}}{2\pi}\int_{\bR} |\xi|^{\frac{\alpha-2}{6}} e^{i\frac{x-(h_n^j)^2\xi_n^j}{h_n^j}\xi-i\frac{t-t_n^j}{(h_n^j)^{\alpha}}|\xi|^{\alpha}} e^{i\frac{(h_n^j)^2\xi_n^j-x_n^j}{h_n^j}\xi}[e^{i(\cdot)h_n^j\xi_n^j}\phi^j]^{\wedge}(\xi)\ddd\xi \notag\\
&\quad \times (h_n^k)^{-\frac{\alpha+1}{6}} \frac{e^{-ix_n^k\xi_n^k}}{2\pi}\int_{\bR} |\xi|^{\frac{\alpha-2}{6}} e^{i\frac{x-(h_n^j)^2\xi_n^k}{h_n^k}\xi-i\frac{t-t_n^k}{(h_n^k)^{\alpha}}|\xi|^{\alpha}} e^{i\frac{(h_n^j)^2\xi_n^k-x_n^k}{h_n^k}\xi}[e^{i(\cdot)h_n^k\xi_n^k}\phi^k]^{\wedge}(\xi)\ddd\xi \notag\\
&=:(h_n^j)^{-\frac{\alpha+1}{6}}\tilde{\Phi}_n^j\l(\frac{t-t_n^j}{(h_n^j)^{\alpha}}, \frac{x-(h_n^j)^2\xi_n^j}{h_n^j}\r) \cdot (h_n^k)^{-\frac{\alpha+1}{6}}\tilde{\Phi}_n^k\l(\frac{t-t_n^k}{(h_n^k)^{\alpha}}, \frac{x-(h_n^j)^2\xi_n^k}{h_n^k}\r). \label{equation_profiles orthogonal_0.5}
\end{align}
Based on the expression \eqref{equation_profiles orthogonal_0.5}, the assumption $h_n^j\sim h_n^k$ gives
\begin{align*}
\|\Phi_n^{j,k}\|_{L_{t,x}^3}& \sim\l\|\tilde{\Phi}_n^j\l(t-\frac{t_n^j}{(h_n^j)^{\alpha}}, x-{h_n^j\xi_n^j}\r) \tilde{\Phi}_n^k\l(\frac{(h_n^j)^{\alpha}t-t_n^k}{(h_n^k)^{\alpha}}, \frac{h_n^j}{h_n^k} (x-h_n^j\xi_n^k)\r)\r\|_{L_{t,x}^{3}}.
\end{align*}
Similarly we still have the following estimate
\[\lim_{R\to\infty}\int_{\{|s|+|y|> R\}} |\tilde{\Phi}_n^j(s,y)|^{6} \ddd y\ddd s=0.\]
By imitating the argument above, we can get the desired result \eqref{Prop_Strichartz-orthogonal profiles_1} too. Hence we turn to the case
\[\lim_{n\to\infty}\l({h_n^j}+{h_n^k}\r)|\xi_n^j-\xi_n^k|=a, \quad a<\infty.\]
Recall the construction of the linear profile decomposition and the label in \eqref{equation_linear profile decomposition_7}. Therefore, due to the conclusion \eqref{Lemma_frequency and scaling_3} in Lemma \ref{Lemma_frequency and scaling}, it remains for us to deal with the case $\beta=\gamma$ in view of the label \eqref{equation_linear profile decomposition_7}. Consequently, we can assume $(h_n^j,\xi_n^j) \equiv(h_n^k, \xi_n^k) \equiv(h_n,\xi_n)$ from now on.

Since the case $\xi_n\equiv 0$ is much easier, recalling the Remark \ref{Remark_space and time translations}, we may further assume
\begin{equation}\label{equation_profiles orthogonal_1.5}
\lim_{n\to\infty}h_n^j\xi_n^j=\infty
\end{equation}
without loss of generality. By changing the variables in $\|\Phi_n^{j,k}\|_{L_{t,x}^3}$, we turn to investigate
\begin{align*}
\tilde{\Phi}_n^{j,k}(t,x)&:= \int_{\bR} |\xi+h_n\xi_n|^{\frac{\alpha-2}{6}} e^{i\l(x+\frac{x_n^j-x_n^k}{h_n}\r)\xi-i\l(t+\frac{t_n^j-t_n^k}{(h_n)^{\alpha}}\r) |\xi+h_n\xi_n|^{\alpha}}\hat{\phi}^k(\xi) \ddd\xi \\
&\quad \times \int_{\bR} |\xi+h_n\xi_n|^{\frac{\alpha-2}{6}}  e^{ix\xi-it|\xi+h_n\xi_n|^{\alpha}}\hat{\phi}^j(\xi)\ddd\xi \\
&=: \int_{\bR} e^{i\l(x+\frac{x_n^j-x_n^k}{h_n}\r)\xi- i\Psi_{n}^{j,k,t}(\xi)} |\xi+h_n\xi_n|^{\frac{\alpha-2}{6}} \hat{\phi}^k(\xi) \ddd\xi \cdot \int_{\bR} e^{ix\xi-it\Psi_n^t(\xi)} |\xi+h_n\xi_n|^{\frac{\alpha-2}{6}}  \hat{\phi}^j(\xi)\ddd\xi\\
&=: A_n^{j,k}(t,x)\cdot B_n^j(t,x).
\end{align*}
To get the desired conclusion \eqref{Prop_Strichartz-orthogonal profiles_1}, it suffices to show that the following estimate
\begin{equation}\label{equation_profiles orthogonal_2}
\lim_{n\to\infty}\l\|A_n^{j,k}(t,x) B_n^j(t,x)\r\|_{L_{t,x}^{3}}=0
\end{equation}
holds for all $j\neq k$. Just as what we have done in Proposition \ref{Prop_dislocation property}, based on the assumption \eqref{equation_profiles orthogonal_1.5}, we can rewrite
\[\Psi_{n}^{j,k,t}(\xi)=\sum_{m=1}^{\infty} a_{n}^{m,j,k,t} (\xi)^{m}, \quad \Psi_n^{t}(\xi)=\sum_{m=1}^{\infty} a_n^{m,t} (\xi)^m.\]
Note that the differences of the coefficients
\begin{equation*}
b_n^{m,j,k}:=a_n^{m,j,k,t}-a_n^{m,t}=\binom{\alpha}{m}\frac{t_n^j-t_n^k}{(h_n)^{\alpha}}\l|h_n\xi_n\r|^{\alpha-m}
\end{equation*}
are independent of $t$ since the difference of the functions
\[\Psi_{n}^{j,k,t}(\xi)-\Psi_n^{t}(\xi)=\frac{t_n^j-t_n^k}{(h_n)^{\alpha}} |\xi+h_n\xi_n|^{\alpha}\]
is independent of $t$. Meanwhile the assumption \eqref{equation_profiles orthogonal_1.5} implies $b_n^{m+1,j,k}\ll b_n^{m,j,k}$ for $n$ large enough. Hence we have, after passing to a subsequence, the following condition
\begin{equation}\label{equation_profiles orthogonal_5}
\lim_{n\to\infty}\l(\l|\tilde{b}_n^{1,j,k}\r|+\l|b_n^{2,j,k}\r|\r)=\infty, \quad \quad \tilde{b}_n^{1,j,k}:=\frac{x_n^k-x_n^j}{h_n}+b_n^{1,j,k}
\end{equation}
due to the limit-orthogonality property \eqref{Prop_linear profile decomposition_3} in Proposition \ref{Prop_linear profile decomposition} and the conditional dislocation property Proposition \ref{Prop_dislocation property}. Again, the method of stationary phase will provide the decay estimates of $A_n^{j,k}(t,x)$ and $B_n^j(t,x)$. Combining the trivial size estimates and the oscillation estimates deduced by the classical van der Corput Lemma, we always have the following estimates
\begin{equation}\label{equation_profiles orthogonal_3}
|A_n^{j,k}(t,x)|\lesssim_{\phi^k} \min\l\{|h_n\xi_n|^{\frac{\alpha-2}{6}}, \l|t-\frac{t_n^k-t_n^j}{(h_n)^{\alpha}}\r|^{-\frac{1}{2}}\!\!\!\!|h_n\xi_n|^{-\frac{\alpha-2}{3}}\r\}
\end{equation}
and
\begin{equation}\label{equation_profiles orthogonal_4}
|B_n^j(t,x)|\lesssim_{\phi^j} \min\l\{|h_n\xi_n|^{\frac{\alpha-2}{6}}, |t|^{-\frac{1}{2}}|h_n\xi_n|^{-\frac{\alpha-2}{3}}\r\}.
\end{equation}
To get the non-stationary bounds, we decompose the spatial space into
\begin{align*}
A_t &:=\l\{x: \l|x+\frac{x_n^j-x_n^k}{h_n}-a_n^{1,j,k,t}\r|\lesssim_{\phi^k} \l|t-\frac{t_n^k-t_n^j}{(h_n)^{\alpha}}\r| |h_n\xi_n|^{\alpha-2}\r\}, \\
B_t &:=\l\{x: \l|x-a_n^{1,t}\r|\lesssim_{\phi^j} |t| |h_n\xi_n|^{\alpha-2}\r\}, \\
C_t &:=\bR\setminus (A_t\cup B_t),
\end{align*}
where the implicit constants in the definitions of $A_t$ and $B_t$ may depend on the Fourier supports of $\phi^k$ and $\phi^j$. Note that
\[A_t=\l\{\l|x-a_n^{1,t}-\tilde{b}_n^{1,j,k}\r|\lesssim_{\phi^k} \l|t-\frac{t_n^k-t_n^j}{(h_n)^{\alpha}}\r| |h_n\xi_n|^{\alpha-2}\r\}\]
by the definition of $\tilde{b}_n^{1,j,k}$. On the other hand if $x\in C_t$, we always have
\[|h_n\xi_n|^{\alpha-2} \l|t-\frac{t_n^k-t_n^j}{(h_n)^{\alpha}}\r| \Big/ \l|x+\frac{x_n^j-x_n^k}{h_n}-a_n^{1,j,k,t}\r| \lesssim_{\phi^k} 1, \quad \quad \frac{|t| |h_n\xi_n|^{\alpha-2}}{\l|x-a_n^{1,t}\r|} \lesssim_{\phi^j} 1.\]
Hence on $\bR\times C_t$, we can use the classical van der Corput Lemma to obtain the non-stationary bounds
\begin{equation}\label{equation_profiles orthogonal_4.5}
|A_n^{j,k}(t,x)|\lesssim_{\phi^k} \frac{|h_n\xi_n|^{\frac{\alpha-2}{6}}}{\l|x-a_n^{1,t}-\tilde{b}_n^{1,j,k}\r|}, \quad |B_n^j(t,x)|\lesssim_{\phi^j} \frac{|h_n\xi_n|^{\frac{\alpha-2}{6}}}{\l|x-a_n^{1,t}\r|}.
\end{equation}

Combining the estimates \eqref{equation_profiles orthogonal_3}, \eqref{equation_profiles orthogonal_4} and \eqref{equation_profiles orthogonal_4.5}, together with the condition \eqref{equation_profiles orthogonal_5}, we can get the desired result \eqref{equation_profiles orthogonal_2} by following an analogous argument in \cite[Lemma 6.1, \textbf{Case 2}]{JPS_2010}.

The details for the remaining proof are very long but essentially the same as \cite[Lemma 6.1, \textbf{Case 2}]{JPS_2010}. For the convenience of the reader and avoiding too much redundancy, we provide part of the details and the rest of the proof will be sketchy.  Split the time space into $\bR=\tau_0^{-}\cup\tau_0\cup\tilde{\tau}_n \cup\tau_n\cup \tau_n^{+}$ where
\[\tau_0:=\l[-|h_n\xi_n|^{2-\alpha}, |h_n\xi_n|^{2-\alpha}\r], \quad \tau_n:=\l[\frac{t_n^k-t_n^j}{(h_n)^{\alpha}}-|h_n\xi_n|^{2-\alpha}, \frac{t_n^k-t_n^j}{(h_n)^{\alpha}}+|h_n\xi_n|^{2-\alpha}\r],\]
and
\[\tau_0^{-}:=(-\infty,-|h_n\xi_n|^{2-\alpha}],\; \tilde{\tau}_n:=\l[|h_n\xi_n|^{2-\alpha}, \frac{t_n^k-t_n^j}{(h_n)^{\alpha}}-|h_n\xi_n|^{2-\alpha}\r],\; \tau_n^{+}:=\l[\frac{t_n^k-t_n^j}{(h_n)^{\alpha}}+|h_n\xi_n|^{2-\alpha},\infty\r)\]
For simplicity we use the notation $I(\tau_0,A_t)$ to denote the integral of $|A_n^{j,k}B_n^j|^3$ on the domain $\tau_0\times A_t$. In other words, we define
\[I(\tau_0,A_t):=\int_{t\in \tau_0}\int_{x\in A_t}|A_n^{j,k}B_n^j|^3 \ddd x\ddd t.\]
Similarly for the notations $I(\tau_n, B_t)$, $I(\tau_n^{+}, C_t)$ and so on. Taking \eqref{equation_profiles orthogonal_5} into consideration, the rest of this proof is divided into two parts.

\textbf{Case A:} If there holds
\begin{equation}\label{equation_profiles orthogonal_6}
\lim_{n\to\infty}|b_n^{2,j,k}|=\lim_{n\to\infty}\l|\frac{t_n^k-t_n^j}{(h_n)^{\alpha}}(h_n\xi_n)^{\alpha-2}\r|=\infty,
\end{equation}
which means that $\frac{|t_n^j-t_n^k|}{(h_n)^{\alpha}}\gg |h_n\xi_n|^{-(\alpha-2)}$ for $n$ large enough. Then the desired result comes from a similar process in the proof of \cite[Lemma 6.1, \textbf{Case 2aI}]{JPS_2010}. By taking a subsequence, together with the symmetry of positive and negative cases, we may assume that $\frac{t_n^k-t_n^j}{(h_n)^{\alpha}}>0$ without loss of generality. From the estimates \eqref{equation_profiles orthogonal_3} and \eqref{equation_profiles orthogonal_4}, it is not hard to show that
\[I(\tau_0, B_t)\lesssim \l|\frac{t_n^k-t_n^j}{(h_n)^{\alpha}}(h_n\xi_n)^{\alpha-2}\r|^{-\frac{3}{2}}, \quad I(\tau_n, B_t)\lesssim \l|\frac{t_n^k-t_n^j}{(h_n)^{\alpha}}(h_n\xi_n)^{\alpha-2}\r|^{-\frac{1}{2}}.\]
And meanwhile some computation gives
\[I(\tau_0^{-}, B_t)\lesssim \l|\frac{t_n^k-t_n^j}{(h_n)^{\alpha}}(h_n\xi_n)^{\alpha-2}\r|^{-\frac{1}{2}}, \quad I(\tau_n^{+},B_t)\lesssim \l|\frac{t_n^k-t_n^j}{(h_n)^{\alpha}}(h_n\xi_n)^{\alpha-2}\r|^{-\frac{1}{2}}.\]
For the term $I(\tilde{\tau}_n, B_t)$, recall the indefinite integral
\[\int t^{-1/2}(a-t)^{-3/2} \ddd t= \frac{2\sqrt{t}}{a\sqrt{a-t}}+C.\]
Therefore by the condition $\frac{|t_n^j-t_n^k|}{(h_n)^{\alpha}}\gg |h_n\xi_n|^{-(\alpha-2)}$ we can obtain
\[I(\tilde{\tau}_n, B_t)\lesssim |h_n\xi_n|^{-(\alpha-2)}\frac{2\sqrt{(t_n^j-t_n^k)/(h_n)^{\alpha}-(h_n\xi_n)^{-\alpha-2}}}{(t_n^j-t_n^k)/(h_n)^{\alpha}\sqrt{(h_n\xi_n)^{-(\alpha-2)}}} \lesssim \l|\frac{t_n^k-t_n^j}{(h_n)^{\alpha}}(h_n\xi_n)^{\alpha-2}\r|^{-\frac{1}{2}}.\]
Hence we conclude that $I(\bR, B_t)\to0$ as $n\to\infty$ due to the condition \eqref{equation_profiles orthogonal_6}. Analogous arguments give the estimate $I(\bR, A_t)\to0$ as $n\to\infty$. Finally, by combining the non-stationary bounds \eqref{equation_profiles orthogonal_4.5} and further splitting $\tilde{\tau}_n$ into
\[\tilde{\tau}_n:=\l[|h_n\xi_n|^{2-\alpha}, \frac{t_n^k-t_n^j}{2(h_n)^{\alpha}}\r] \bigcup \l[\frac{t_n^k-t_n^j}{2(h_n)^{\alpha}}, \frac{t_n^k-t_n^j}{(h_n)^{\alpha}}-|h_n\xi_n|^{2-\alpha}\r],\]
we can obtain the estimate $I(\bR, C_t)\to0$ as $n\to\infty$ and finish the proof of this case.

\textbf{Case B:} If $|b_n^{2,j,k}|\leq C_0$ for some fixed $C_0>0$ and
\[\lim_{n\to\infty}\l|\tilde{b}_n^{1,j,k}\r|=\lim_{n\to\infty} \l|\frac{x_n^j-x_n^k-\alpha(t_n^j-t_n^k)(\xi_n)^{\alpha-1}}{h_n}\r|=\infty.\]
Analogously the desired result comes from a similar process in the proof of \cite[Lemma 6.1, \textbf{Case 2aII}]{JPS_2010}. We may assume $\tilde{b}_n^{1,j,k}>0$ at first. In this case, the corresponding decomposition for the time space is $\bR:=\dot{\tau}_K^{+}\cup\dot{\tau}_K^{-}$ for a large constant $K\gg C_0$, where
\[\dot{\tau}_K^{+}:=\l\{t: (h_n\xi_n)^{\alpha-2}|t|\geq K\r\}.\]
Note that on $\dot{\tau}_K^{+}$ it holds $|(t_n^k-t_n^j)/(h_n)^{\alpha}|\ll |t|$ and on $\dot{\tau}_K^{-}\times B_t$ it holds
\[|x-\alpha (h_n\xi_n)^{\alpha-1}|\ll \l|\tilde{b}_n^{1,j,k}\r|\]
for $n$ large enough. These two facts together with the stationary bounds \eqref{equation_profiles orthogonal_3} and \eqref{equation_profiles orthogonal_4} imply
\[I(\dot{\tau}_K^{+}, B_t)\lesssim K^{-1}, \quad I(\dot{\tau}_K^{-}, B_t)\lesssim K^{1/2}\l|\tilde{b}_n^{1,j,k}\r|^{-3}.\]
The first estimate for $I(\dot{\tau}_K^{+}, B_t)$ is uniform in all large $n$ and is going to zero as $K$ goes to infinity; while the second estimate for $I(\dot{\tau}_K^{-}, B_t)$ is going to zero as $n$ goes to infinity. Thereby, after a similar argument for $I(\bR, A_t)$, we can obtain the following two estimates
\[\lim_{n\to\infty}I(\bR, A_t)=0, \quad \lim_{n\to\infty} I(\bR, B_t)=0.\]
For the term $I(\bR, C_t)$, we should use the non-stationary bounds \eqref{equation_profiles orthogonal_4.5} too. The result $I(\dot{\tau}_K^{+}, C_t)\lesssim K^{-1}$ is not hard to obtain. Hence it remains for us to estimate $I(\dot{\tau}_K^{-}, C_t)$. Just as what we have done in \textbf{Case A}, for $t\in \dot{\tau}_K^{-}$ if we split $C_t$ further into $C_t=B_t^{-}\cup B_t^{+}\cup A_t^{-}\cup A_t^{+}$ for $n$ large enough where
\[B_t^{-}:=\l(-\infty, a_n^{1,t}-|t||h_n\xi_n|^{\alpha-2}\r], \quad B_t^{+}:=\l[a_n^{1,t}+|t||h_n\xi_n|^{\alpha-2}, a_n^{1,t}+\frac{\tilde{b}_n^{1,j,k}}{2}\r],\]
and
\[A_t^{-}:=\l[a_n^{1,t}+\frac{\tilde{b}_n^{1,j,k}}{2}, a_n^{1,t}+\tilde{b}_n^{1,j,k}-\l|t-\frac{t_n^k-t_n^j}{(h_n)^{\alpha}}\r||h_n\xi_n|^{\alpha-2}\r],\]
with
\[A_t^{+}:=\l[ a_n^{1,t}+\tilde{b}_n^{1,j,k}+\l|t-\frac{t_n^k-t_n^j}{(h_n)^{\alpha}}\r||h_n\xi_n|^{\alpha-2}, +\infty\r),\]
then we can get the estimate $I(\bR,C_t)\to 0$ as $n\to\infty$ by estimating the integral piece by piece. Therefore, we conclude the desired result \eqref{equation_profiles orthogonal_2} and finish the proof.
\end{proof}

\section{Extremals for symmetric alpha-Strichartz estimate}\label{Sec_Symmetric alpha-Strichartz extremals}
With the linear profile decomposition Proposition \ref{Prop_linear profile decomposition}, the Strichartz-orthogonality of profiles Proposition \ref{Prop_Strichartz-orthogonal profiles} and the following asymptotic \schrodinger behavior Lemma \ref{Lemma_asymptotic schrodinger} in place, we are ready to give the proof of the desired extremal result Theorem \ref{Thm_extremals-symmetric} for symmetric $\alpha$-Strichartz estimate. This arguments can be directly used in asymmetric cases, which will be shown in Section \ref{Sec_Asymmetric alpha-Strichartz} later.

\begin{lemma}[Asymptotic \schrodinger behavior]\label{Lemma_asymptotic schrodinger}
If $\|\phi\|_{L_x^2(\bR)}=1$ and $\lim_{n\to\infty}|\xi_n|=\infty$, then we have
\begin{equation}\label{Lemma_asymptotic schrodinger_1}
\lim_{n\to\infty}\l\|[D^{\frac{\alpha-2}{6}}][e^{it|\nabla|^{\alpha}}][e^{i(\cdot)\xi_n}\phi]\r\|_{L_{t,x}^{6}(\bR^{2})}
=\l(\frac{\alpha^2-\alpha}{2}\r)^{-\frac{1}{6}}\l\|[e^{-it\Delta}]\phi\r\|_{L_{t,x}^{6}(\bR^{2})}.
\end{equation}
\end{lemma}

\begin{proof}[\textbf{Proof of Lemma \ref{Lemma_asymptotic schrodinger}}]
This is a standard consequence by changing of variables and the dominated convergence theorem. While the dominating function, which can be chosen as
\begin{equation*}
F(t,x):=\l\{\begin{array}{cc}
C_{\phi}[(1+|t|)(1+|x|)]^{-1/4}, & |x|\lesssim_{\phi} |t|;\\
C_{\phi}[(1+|t|)(1+|x|)]^{-1/2}, & |x|\gtrsim_{\phi} |t|,
\end{array}\r.
\end{equation*}
also comes from the classical van der Corput Lemma. The readers can see \cite[Proposition 7.1]{JPS_2010} or \cite[Remark 1.7]{Shao_2009} for more details about the very similar arguments in other contexts. This idea can also be seen in some earlier papers \cite{CCT_2003} and \cite{Tao_2007}. Note that the assumption $\lim_{n\to\infty}|\xi_n|=\infty$ can guarantee the uniform convergency when we use a series which is similar to \eqref{equation_dislocations_5}. We omit the proof here.
\end{proof}

\begin{remark}\label{Remark_asymptotic schrodinger}
Based on the existence of extremals for $\mathbf{M}_2$, Lemma \ref{Lemma_asymptotic schrodinger} implies
\[\mathbf{M}_{\alpha}\geq \l(\frac{\alpha^2-\alpha}{2}\r)^{-\frac{1}{6}}\mathbf{M}_2.\]
Meanwhile we pointed out that the non-precompactness, in view of Theorem \ref{Thm_extremals-symmetric}, is different from the non-existence of extremals. Indeed when $\alpha=2$, Foschi \cite{Foschi_2007} and Hundertmark-Zharnitsky \cite{HZ_2006} independently show that the sharp constant $\mathbf{M}_2=12^{-12}$ and the only extremals are Gaussians up to symmetries. Based on this result, Lemma \ref{Lemma_asymptotic schrodinger} implies
\[\mathbf{M}_{\alpha}\geq [\sqrt{3}\alpha(\alpha-1)]^{-\frac{1}{6}}.\]
\end{remark}

\begin{proof}[\textbf{Proof of Theorem \ref{Thm_extremals-symmetric}}]
Let $\{u_n\}_{n\geq1}$ be an extremal sequence for $\mathbf{M}_{\alpha}$. Then, up to subsequences, by the profile decomposition Proposition \ref{Prop_linear profile decomposition} we can decompose $u_n$ into linear profiles as
\[u_n=\sum_{j=1}^{J} [T_n^j]\phi^j+\omega_n^{J}.\]
Due to the vanishing Strichartz norm estimate \eqref{Prop_linear profile decomposition_2} for the remainder term in Proposition \ref{Prop_linear profile decomposition}, we obtain that for arbitrary $\epsilon>0$, there exists $N_{\epsilon}$ such that for all $N\geq N_{\epsilon}$ and all $n\geq N_{\epsilon}$,
\[\mathbf{M}_{\alpha}-\epsilon\leq \l\| D^{\frac{\alpha-2}{6}}\sum_{j=1}^{N} [e^{it|\nabla|^{\alpha}}][T_n^j]\phi^j\r\|_{L_{t,x}^{6}}.\]
Hence the Strichartz-orthogonality of profiles Proposition \ref{Prop_Strichartz-orthogonal profiles} gives the following
\[\mathbf{M}_{\alpha}^6-C_{\alpha}\epsilon\leq \sum_{j=1}^{N}\l\| [D^{\frac{\alpha-2}{6}}][e^{it|\nabla|^{\alpha}}][T_n^j]\phi^j\r\|^6_{L_{t,x}^6}.\]
Take $j_0$ by
\[\l\|[D^{\frac{\alpha-2}{6}}][e^{it|\nabla|^{\alpha}}][T_n^{j_0}]\phi^{j_0}\r\|_{L_{t,x}^6}=\max\l\{\l\| [D^{\frac{\alpha-2}{6}}][e^{it|\nabla|^{\alpha}}][T_n^j]\phi^j\r\|_{L_{t,x}^6}: 1\leq j\leq N_0\r\}.\]
The $\alpha$-Strichartz estimate \eqref{equation_symmetric alpha-Strichartz} implies
\begin{align}
\mathbf{M}_{\alpha}^6-C_{\alpha}\epsilon &\leq \sum_{j=1}^{N}\l\| [D^{\frac{\alpha-2}{6}}] [e^{it|\nabla|^{\alpha}}][T_n^j]\phi^j \r\|^{6}_{L_{t,x}^{6}}\notag \\
&\leq \mathbf{M}_{\alpha}^6\sum_{j=1}^{N} \|\phi^j\|_{L_x^2}^6 \leq \mathbf{M}_{\alpha}^6\l(\sum_{j=1}^{N} \|\phi^j\|_{L_x^2}^2\r)^3
\leq \mathbf{M}_{\alpha}^6, \label{equation_extremals_2}
\end{align}
where the last inequality comes from the fact that the $L^2$-orthogonal identity \eqref{Prop_linear profile decomposition_4} which leads to
\begin{equation}\label{equation_extremals_3}
\sum_{j=1}^{\infty}\|\phi^j\|_{L_x^2}^2\leq \lim_{n\to\infty}\|u_n\|^2_{L_x^2}=1.
\end{equation}
This fact also deduces $\lim_{j\to\infty}\|\phi^j\|_{L_x^2}=0$. Consequently, we can choose $j_0$ independent of $\epsilon$. Hence we may let $\epsilon\to 0$ in \eqref{equation_extremals_2} to obtain
\[\|\phi^{j_0}\|_{L_x^2}=1,\]
which means $\phi^j=0$ for all $j\neq j_0$ due to the inequality \eqref{equation_extremals_3}. Therefore by the linear profile decomposition we conclude
\[\lim_{n\to\infty}\l\| [D^{\frac{\alpha-2}{6}}] [e^{it|\nabla|^{\alpha}}][T_n^{j_0}]\phi^{j_0} \r\|_{L_{t,x}^{6}}=\lim_{n\to\infty}\l\| [D^{\frac{\alpha-2}{6}}][e^{it|\nabla|^{\alpha}}][e^{i(\cdot)h_n^{j_0}\xi_n^{j_0}}]\phi^{j_0} \r\|_{L_{t,x}^{6}}=\mathbf{M}_{\alpha}.\]
This suggests that $e^{ixh_n^{j_0}\xi_n^{j_0}}\phi^{j_0}(x)\in L_x^2$ is an extremal sequence where either $\lim_{n\to\infty}|h_n^{j_0}\xi_n^{j_0}| \to\infty$ or $h_n^{j_0}\xi_n^{j_0}\equiv0$. For the case $h_n^{j_0}\xi_n^{j_0}\equiv0$, we get the desired extremal function $\phi^{j_0}$. However for the case $|h_n^{j_0}\xi_n^{j_0}|\to\infty$, we should do more investigation as follows. Indeed as we will see, Lemma \ref{Lemma_asymptotic schrodinger} gives the desired conclusion and finishes the proof.

If we have the strict inequality \eqref{Thm_extremals-symmetric_1}, then the case $|h_n^{j_0}\xi_n^{j_0}|\to\infty$ is ruled out by Lemma \ref{Lemma_asymptotic schrodinger} and hence all the extremal sequences for $\mathbf{M}_{\alpha}$ are precompact up to symmetries. On the other hand if
\[\mathbf{M}_{\alpha}=[\sqrt{3}\alpha(\alpha-1)]^{-\frac{1}{6}},\]
then, after normalizing, $\tilde{u}_n(x):=\sqrt{n}e^{in^2x}e^{-|n(x-x_0)|^2}$ will give an extremal sequence that is not precompact up to symmetries, since $\tilde{u}_n$ goes to zero up to symmetries in the weak topology of $L^2(\bR)$. Meanwhile, it is easy to check that $(\tilde{u}_n)$ concentrates at $x_0$.
\end{proof}

\section{Extremals for asymmetric alpha-Strichartz estimate} \label{Sec_Asymmetric alpha-Strichartz}
As what we have stated before, our method can produce some extremal results for the asymmetric $\alpha$-Strichartz estimates \eqref{equation_asymmetric alpha-Strichartz} as well. This mainly because that the profile decomposition Proposition \ref{Prop_linear profile decomposition} is simultaneously equipped with the Strichartz-orthogonality Proposition \ref{Prop_Strichartz-orthogonal profiles} for all the profiles. By imitating the arguments in the proof of Theorem \ref{Thm_extremals-symmetric}, it not hard to see that the desired asymmetric $\alpha$-Strichartz result Theorem \ref{Thm_extremals asymmetric} is a standard consequence of the following two lemmas which are generalizations of the estimates \eqref{Prop_Strichartz-orthogonal profiles_2} and \eqref{Lemma_asymptotic schrodinger_1} respectively.

\begin{lemma}\label{Lemma_asymmetric orthogonal}
Let $(q,r)$ be non-endpoint pairs and $\tilde{N}\geq1$. For the profiles in Proposition \ref{Prop_linear profile decomposition}, if $q\geq r$ then
\begin{equation*}
\lim_{n\to\infty} \l\|\sum_{j=1}^{\tilde{N}}[D^{\frac{\alpha-2}{q}}][e^{it|\nabla|^{\alpha}}][T_n^j]\phi^j\r\|_{L_t^q L_x^r}^r \leq \sum_{j=1}^{\tilde{N}}\lim_{n\to\infty} \l\|[D^{\frac{\alpha-2}{q}}][e^{it|\nabla|^{\alpha}}][T_n^j]\phi^j\r\|_{L_t^q L_x^r}^r;
\end{equation*}
meanwhile if $q\leq r$ then
\begin{equation*}
\lim_{n\to\infty} \l\|\sum_{j=1}^{\tilde{N}}[D^{\frac{\alpha-2}{q}}][e^{it|\nabla|^{\alpha}}][T_n^j]\phi^j\r\|_{L_t^q L_x^r}^q \leq \sum_{j=1}^{\tilde{N}}\lim_{n\to\infty} \l\|[D^{\frac{\alpha-2}{q}}][e^{it|\nabla|^{\alpha}}][T_n^j]\phi^j\r\|_{L_t^q L_x^r}^q.
\end{equation*}
\end{lemma}

\begin{lemma}\label{Lemma_asymmetric asymptotic schrodinger}
If $\|\phi\|_{L_x^2(\bR)}=1$ and $\lim_{n\to\infty}|\xi_n|=\infty$, then we have
\begin{equation*}
\lim_{n\to\infty}\l\|[D^{\frac{\alpha-2}{q}}][e^{it|\nabla|^{\alpha}}][e^{i(\cdot)\xi_n}\phi]\r\|_{L_t^q L_x^r}
=\l(\frac{\alpha^2-\alpha}{2}\r)^{-\frac{1}{q}}\l\|[e^{-it\Delta}]\phi\r\|_{L_t^q L_x^r}.
\end{equation*}
\end{lemma}

\begin{remark}\label{Remark_asymptotic schrodinger_asymmetric}
As in Remark \ref{Remark_asymptotic schrodinger} the non-precompactness up to symmetries, equivalently the equality
\[\tilde{\mathbf{M}}_{\alpha, q,r}=\l(\frac{\alpha^2-\alpha}{2}\r)^{-\frac{1}{q}} \tilde{\mathbf{M}}_{2,q,r},\]
does not mean the non-existence of extremals. For the sharp constant $\tilde{\mathbf{M}}_{2,q,r}$ with asymmetric pairs $(q,r)$, recalling that the existence of extremals have been proved in \cite{Shao_2009_EJDE}, the known results are pretty few even if there are many excellent works such as \cite{BBCH_2009, Carneiro_2009, Goncalves_2019}. Indeed, up to now, we are only aware of the case $\tilde{\mathbf{M}}_{2,8,4}=2^{-1/4}$ as shown in the aforementioned three papers. As far as we know, there is no higher dimensional result for the asymmetric sharp constant $\tilde{\mathbf{M}}_{2,q,r}$.
\end{remark}

We are not planing to show all the detailed proofs of these two lemmas since the arguments are standard. Instead, there will be some useful references for the readers who are interested in the further details. For the Lemma \ref{Lemma_asymmetric orthogonal}, indeed it is a corollary of the Strichartz-orthogonality estimate \eqref{Prop_Strichartz-orthogonal profiles_1}. This fact may be not as obvious as getting the estimate \eqref{Prop_Strichartz-orthogonal profiles_2} from \eqref{Prop_Strichartz-orthogonal profiles_1}, since $q\neq r$ and they may be not natural numbers at all. However this difficult has been overcame by using the interpolation arguments and some \textit{floor function} techniques, see \cite[Lemma 1.6]{Shao_2009_EJDE} for more details. Also notice that the conclusions in Lemma \ref{Lemma_asymmetric orthogonal} are inequalities instead of equalities compared with the estimate \eqref{Prop_Strichartz-orthogonal profiles_2}.

While Lemma \ref{Lemma_asymmetric asymptotic schrodinger} may seem easy to accept since it is obvious an asymmetric generalization of Lemma \ref{Lemma_asymptotic schrodinger}. Indeed here we only need to give the $L_t^qL_x^r$-dominating function
\begin{equation*}
F(t,x):=\l\{\begin{array}{cc}
C_{\phi}(1+|t|)^{-\frac{3}{2q}}(1+|x|)^{-\frac{q-3}{2q}}, & |x|\lesssim_{\phi} |t|;\\
C_{\phi}(1+|t|)^{-\frac{3}{q}}(1+|x|)^{-\frac{q-3}{q}}, & |x|\gtrsim_{\phi} |t|.
\end{array}\r.
\end{equation*}
As what we have shown in the proof of Lemma \ref{Lemma_asymptotic schrodinger}, the detailed arguments can be founded in \cite[Proposition 7.1]{JPS_2010} and \cite[Remark 1.7]{Shao_2009}, or some earlier papers such as \cite{CCT_2003} and \cite{Tao_2007}.

\section{Extremals for non-endpoint alpha-Strichartz estimate}\label{Sec_Extremals nonendpoint alpha-Strichartz}
In this section, we provide the proof of Theorem \ref{Thm_extremals nonendpoint alpha-Strichartz} by following the arguments in \cite{HS_2012}. It is obvious that we only need to prove the following profile decomposition of non-endpoint $\alpha$-Strichartz version Proposition \ref{Prop_nonendpoint linear profile decomposition}, which indeed is a standard consequence of the aforementioned linear profile decomposition results Proposition \ref{Prop_linear profile decomposition} and Proposition \ref{Prop_Strichartz-orthogonal profiles}.
\begin{prop}\label{Prop_nonendpoint linear profile decomposition}
Let $(u_n)$ be a bounded sequence in $L^2(\bR)$. Then, up to subsequences, there exist a sequence of operators $[T_n^j]$ defined by
\[[T_n^{j}]\phi(x):=[e^{-it_n^j|\nabla|^{\alpha}}]\l[(h_n^j)^{-\frac{1}{2}}\phi\l(\frac{x-x_n^j}{h_n^j}\r)\r]\]
with $(h_n^j, x_n^j, t_n^j) \in \bR_{+}\times\bR\times\bR$ and a sequence of functions $\phi^j\in L^2(\bR)$ such that for every $J\geq1$, we have the profile decomposition
\begin{equation*}
u_n=\sum_{j=1}^{J} [T_n^j]\phi^j+\omega_n^{J},
\end{equation*}
where the decomposition possesses the following properties: firstly the remainder term $\omega_n^{J}$ has vanishing Strichartz norm
\begin{equation}\label{Prop_nonendpoint linear profile decomposition_2}
\lim_{J\to\infty}\limsup_{n\to\infty}\l\|[e^{it|\nabla|^{\alpha}}]\omega_n^{J}\r\|_{L_{t,x}^{2\alpha+2}(\bR^{2})}=0;
\end{equation}
secondly the sequence of operators $[T_n^j]$ satisfies that if $j\neq k$, there holds the limit-orthogonality property
\begin{equation}\label{Prop_nonendpoint linear profile decomposition_3}
[T_n^k]^{-1}[T_n^j]\rightharpoonup 0
\end{equation}
as $n$ goes to infinity in the weak operator topology of $\mathcal{B}(L^2)$; for each $J\geq1$, we have
\begin{equation*}
\lim_{n\to\infty}\l[\|u_n\|_{L^2(\bR)}^2-\l(\sum_{j=1}^{J}\|\phi^j\|_{L^2(\bR)}^2\r)-\|\omega_n^{J}\|_{L^2(\bR)}^2\r]=0;
\end{equation*}
moreover for every $j\neq k$ there holds the Strichartz-orthogonality of profiles
\begin{equation}\label{Prop_nonendpoint Strichartz-orthogonal profiles_1}
\lim_{n\to\infty}\l\|[e^{it|\nabla|^{\alpha}}][T_n^j]\phi^j \cdot [[e^{it|\nabla|^{\alpha}}][T_n^k]\phi^k\r\|_{L_{t,x}^{\alpha+1}(\bR^{2})}=0.
\end{equation}
\end{prop}

\begin{remark}
In this case, without the frequency parameters $\xi_n$, the limit-orthogonality property \eqref{Prop_nonendpoint linear profile decomposition_3} holds up to subsequences if and only if
\[\limsup_{n\to\infty}\l(\frac{h_n^j}{h_n^k}+\frac{h_n^k}{h_n^j} +\frac{|t_n^j-t_n^k|}{(h_n^j)^{\alpha}} + \frac{|t_n^j-t_n^k|}{(h_n^k)^{\alpha}} +\frac{|x_n^j-x_n^k|}{h_n^j}+\frac{|x_n^j-x_n^k|}{h_n^k}\r)=\infty.\]
This conclusion can be seen in the proof of the conditional dislocation property Proposition \ref{Prop_dislocation property}. Note that the condition above is symmetric in the indices $j$ and $k$.
\end{remark}

\begin{proof}[\textbf{Proof of Proposition \ref{Prop_nonendpoint linear profile decomposition}}]
It is not hard to see that the vanishing norm estimate \eqref{Prop_nonendpoint linear profile decomposition_2} follows from the remainder term estimate \eqref{Prop_linear profile decomposition_2} in Proposition \ref{Prop_linear profile decomposition} and Sobolev inequalities. To eliminate the frequency parameters, as shown in the proof of \cite[Theorem 2.4]{HS_2012}, the key point is to deduce the following estimate
\begin{equation}\label{equation_nonendpoint linear profile decomposition_1}
\lim_{|\xi_n|\to\infty}\l\|[e^{it|\nabla|^{\alpha}}][e^{i(\cdot)\xi_n}\phi]\r\|_{L_{t,x}^{2\alpha+2}}=0.
\end{equation}
Then the highly oscillatory terms, which mean the terms $[T_n^j]\phi^j(x)$ with
\[\lim_{n\to\infty}|h_n^j\xi_n^j|=\infty,\]
in Proposition \ref{Prop_linear profile decomposition} can be reorganized into the remainder term. After that, the desired Strichartz-orthogonality \eqref{Prop_nonendpoint Strichartz-orthogonal profiles_1} of these profiles is much easier to established due to the lack of frequency parameters, see also \cite[Lemma 2.7]{HS_2012} for further details. Other conclusions come from Proposition \ref{Prop_linear profile decomposition} and Proposition \ref{Prop_Strichartz-orthogonal profiles} accordingly.

To obtain the estimate \eqref{equation_nonendpoint linear profile decomposition_1}, we can follow similar arguments in the proof of Lemma \ref{Lemma_asymptotic schrodinger}, see also \cite[Theorem 2.4]{HS_2012}. Indeed we could rewrite
\[\l|[e^{it|\nabla|^{\alpha}}][e^{i(\cdot)\xi_n}\phi](x)\r|=\frac{1}{2\pi} \l|\int_{\bR} e^{ix\xi +it\dot{\Phi}_n(\xi)} \hat{\phi}(\xi)\ddd\xi\r|\]
with
\[\dot{\Phi}_n(\xi):=\sum_{m=1}^{\infty} \binom{\alpha}{m} (\xi_n)^{\alpha-m} \xi^m.\]
Notice that by density we can assume $\phi$ to be a Schwartz function with compact Fourier support and the assumption $|\xi_n|\to\infty$ can guarantee the uniform convergency of $\dot{\Phi}_n$ for $n$ large enough. Hence by changing of variables and van der Corput Lemma, the dominated convergence theorem implies the desired conclusion \eqref{equation_nonendpoint linear profile decomposition_1}.
\end{proof}

\bigskip
\subsection*{Acknowledgements}
The authors would like to thank Shuanglin Shao for his valuable conversations. This work was supported by National Natural Science Foundation of China [grant numbers 11871452, 12071052]. The first author acknowledges the support from UCAS Joint Training Program, and this research was completed when the first author visited University of Kansas whose hospitality is also appreciated.

\bigskip
\begin{flushleft}
\textsc{Boning Di \\
School of Mathematical Sciences \\
University of Chinese Academy of Sciences \\
Beijing, 100049 \\ P.R. China}

E-mail: \textsf{diboning18@mails.ucas.ac.cn}
\end{flushleft}

\begin{flushleft}
\textsc{Dunyan Yan \\
School of Mathematical Sciences \\
University of Chinese Academy of Sciences \\
Beijing, 100049 \\ P.R. China}

E-mail: \textsf{ydunyan@ucas.ac.cn}
\end{flushleft}

\end{document}